\documentclass[10pt]{article}

\usepackage{amsmath, amsthm, amssymb, amsfonts, enumerate}
\usepackage{mathtools}
\usepackage[colorlinks=true,linkcolor=blue,urlcolor=blue]{hyperref}
\usepackage{dsfont}
\usepackage{color}
\usepackage{geometry}
\usepackage{todonotes}
\usepackage{bm}

\newtheorem{theorem}{Theorem}[section]
\newtheorem{remark}[theorem]{Remark}

\newtheorem{lemma}[theorem]{Lemma}
\newtheorem{proposition}[theorem]{Proposition}

\newtheorem{definition}[theorem]{Definition}

\numberwithin{equation}{section}

\begin{document}
\title{Multi-dimensional reflected BSDEs driven by $G$-Brownian motion with diagonal generators}
\author{Hanwu Li\textsuperscript{a,b}\ \ and \ \ Guomin Liu\textsuperscript{c,}\footnote{Corresponding author.\\
		\indent\ \ \textit{E-mail addresses:} lihanwu@sdu.edu.cn (H. Li),   gmliu@nankai.edu.cn (G. Liu).}\\	
	\\\footnotesize\textsuperscript{a}\textit{Research Center for Mathematics and Interdisciplinary Sciences, Shandong University,
		 Qingdao 266237, China}
	\\\noindent{\footnotesize\noindent\textsuperscript{b}\textit{Frontiers Science Center for Nonlinear Expectations (Ministry of Education), Shandong} }\\{\footnotesize \textit{University, Qingdao 266237, China}}
	\\\footnotesize\textsuperscript{c}\textit{School of Mathematical Sciences, Nankai University, Tianjin 300071, China}}
\date{}
\maketitle
\begin{abstract}
We consider the well-posedness problem of multi-dimensional reflected backward
stochastic differential equations driven by $G$-Brownian motion ($G$-BSDEs)
with diagonal generators. Two methods, including the penalization
method and the Picard iteration argument, are provided to prove the existence and
uniqueness of the solutions. We also study its connection with the obstacle
problem of a system of fully nonlinear PDEs.
\end{abstract}

\smallskip

{\textbf{Keywords}}: $G$-expectation, $G$-Brownian motion, multi-dimensional
BSDEs, fully nonlinear PDEs.

\smallskip

{\textbf{MSC2010 subject classification}}: 60H10; 60H30

\smallskip


\section{Introduction}

Pardoux and Peng \cite{PP} first introduced the nonlinear backward
stochastic differential equations (BSDEs) taking the following form
\[
Y_{t}=\xi+\int_{t}^{T}f(s,Y_{s},Z_{s})ds-\int_{t}^{T}Z_{s}dB_{s},
\]
where the solution is a pair of adapted processes $(Y,Z)$.
The BSDE theory has attracted a great deal of attention since it has wide
applications in partial differential equations (PDEs), stochastic control and
financial mathematics (e.g., \cite{EPQ97,PP92,P93}). One of the
most important extensions of BSDE theory is the reflected BSDEs (see
\cite{KKPPQ}), which means that the solution $Y$ of the equation is forced to
stay above a given process, called the obstacle. To this end, a non-decreasing
process should be added to push the solution upward, which behaves in a
minimal way such that it satisfies the Skorohod condition. For some further
developments, one may refer to \cite{BCFE,EPQ,GIOOQ,KLQT,LX,PX,WX} and the references therein.

It is worth pointing out that all the papers listed above are considered under the classical expectation framework. Recently, Peng systematically introduced a time-consistent nonlinear
expectation theory, i.e., the $G$-expectation theory, see \cite{P07a,P08a,P10}%
, which is a useful tool to study the financial problems under volatility
uncertainty and the probabilistic representation for fully nonlinear PDEs. In this framework, a new type of
Brownian motion with independent and stationary increments, called
$G$-Brownian motion, was constructed and the corresponding It\^{o}'s integral
was established. Then, Hu et al. \cite{HJPS1} investigated the following type of
backward stochastic differential equations driven by $G$-Brownian motion
($G$-BSDEs):
\[
Y_{t}=\xi+\int_{t}^{T}f(s,Y_{s},Z_{s})ds+\int_{t}^{T}g(s,Y_{s},Z_{s})d\langle
B\rangle_{s}-\int_{t}^{T}Z_{s}dB_{s}-(K_{T}-K_{t}).
\]
Compared with the classical case, there exists an additional non-increasing
$G$-martingale $K$ in this equation and the quadratic variation process of $B$
is not deterministic. The authors obtained the existence and uniqueness result
of the above $G$-BSDE. The comparison theorem, Feymann-Kac formula and Girsanov
transformation were established in the companion paper \cite{HJPS2}. One important feature is that the solution $Y$ in \cite{HJPS1} is required to be
one-dimensional. Liu \cite{Liu-stochastics} extended the results to the
multi-dimensional case where the generators are assumed to be diagonal with respect to the $z$-term. After that, Hu et al. \cite{HTW2022} consider the multi-dimensional case with diagonally quadratic generators.

The reflected $G$-BSDE with a lower obstacle $S$ has been established by Li, Peng and Soumana Hima \cite{LPSH}, where the solution is a triple of processes $(Y,Z,A)$ with dynamics 
\[
Y_{t}=\xi+\int_{t}^{T}f(s,Y_{s},Z_{s})ds+\int_{t}^{T}g(s,Y_{s},Z_{s})d\langle
B\rangle_{s}-\int_{t}^{T}Z_{s}dB_{s}+(A_{T}-A_{t}),
\]
such that $Y_{t}\geq S_{t}$ and $\{-\int_{0}^{t}(Y_{s}-S_{s})dA_{s}\}_{t\in\lbrack0,T]}$ is a
non-increasing $G$-martingale. In fact, the non-decreasing process $A$ can be
regarded as the discrepancy between the non-decreasing process aiming to push
the solution $Y$ upward and the non-increasing $G$-martingale appearing in
$G$-BSDEs. Due to the appearance of the non-increasing $G$-martingale, the process $A$ in the reflected $G$-BSDE with an upper obstacle is not monotone, which makes this kind of reflected $G$-BSDEs significantly different from the lower obstacle case. Applying a variant comparison theorem, Li and Peng \cite{lp} shows that the solution constructed by the penalization method is the maximal one. For the double obstacles case, i.e., the solution lies between two given processes, we may refer to the paper \cite{LS}. It should be pointed out that the first component $Y$ of reflected $G$-BSDEs in \cite{lp,LPSH,LS} is one-dimensional.

The objective of this paper is to investigate the multi-dimensional reflected $G$-BSDEs. The main difficulty in the
multi-dimensional case is that, due to the nonlinearity of $G$-expectation, the linear combination of $G$-martingales is no longer a $G$-martingale. To overcome this difficulty, we make use of a diagonal structure so that each component of the equation can be independently considered under appropriate circumstances. More precisely, let $x^{i}$ be the $i$-th component of a vector
$x\in\mathbb{R}^{k}$, $i=1,\cdots,k$. The multi-dimensional reflected $G$-BSDE
is of the following type: for  $i=1,\cdots,k$,
\[%
\begin{cases}
Y_{t}^{i}=\xi^{i}+\int_{t}^{T}f^{i}(s,Y_{s},Z_{s}^{i})ds+\int_{t}^{T}%
g^i(s,Y_{s},Z_{s}^{i})d\langle B\rangle_{s}-\int_{t}^{T}Z_{s}^{i}dB_{s}%
+(A_{T}^{i}-A_{t}^{i}),\\
Y_{t}^{i}\geq S_{t}^{i},\ 0\leq t\leq T,\\
\{-\int_{0}^{t}(Y_{s}^{i}-S_{s}^{i})dA_{s}^{i}\}_{t\in\lbrack0,T]}\text{ is a
non-increasing $G$-martingale,}%
\end{cases}
\]
where $Y=(Y^{1},\cdots,Y^{k})^{T}$. Note that the $z$-term of the $i$-th
components of generators $f^{i}$, $g^{i}$ here only depend on $z^{i}$, which is
the point where \textit{diagonal} refers to.

The existence and uniqueness result is established by two different approaches.
The first one provided in Subsection 3.2 is the penalization method, a scheme that is frequently used for constructing the solutions of reflected equations. By a suitable modification to our
new situation, this construction is still valid. Roughly speaking, the $Y$-term
of the solutions to multi-dimensional reflected $G$-BSDEs can be approximated by a
family  of solutions to multi-dimensional $G$-BSDEs. However,  unlike the one-dimensional case, the comparison theorem does not apply in general since a kind of structure condition is needed for the multi-dimensional case (see Theorems \ref{Myth2-3} and \ref{Myth4.1}). To tackle this, we shall make use of a method of linearization to prove the convergence of the approximation sequence. The second one given in Subsection 3.3
is the Picard iteration method motivated by the study of the standard multi-dimensional
$G$-BSDEs in \cite{Liu-stochastics}. We make use of a contraction argument for
the $Y$-term to derive the local well-posedness result, in virtue of the a
priori estimates for  reflected $G$-BSDEs. The global situation is then
obtained by a backward iteration of the local ones. As an application, we
provide a probabilistic representation for solutions of a system of fully
nonlinear PDEs with obstacle constraints, based on the construction via penalization.

This paper is organized as follows. In Section 2, we recall some basic notions of $G$-expectation,  multi-dimensional $G$-BSDEs and one-dimensional reflected $G$-BSDEs. Section 3 is devoted to the study of  well-posedness of multi-dimensional reflected $G$-BSDEs, based on the Picard iteration and the penalization method, respectively. We formulate our probabilistic representation for fully nonlinear PDE systems with obstacles in Section 4.

\section{Preliminaries}

We first recall some basic results about $G$-expectation, multi-dimensional
$G$-BSDEs and one-dimensional reflected $G$-BSDEs, which are needed in the
sequel  and the readers may refer to the papers
\cite{DHP11,LPSH,LS,Liu-stochastics,P07a,P08a,P10,S11}
 for more details.
For convenience, every element $x \in\mathbb{R}^{k}$ is identified as a column
vector with $l$-th component $x^{l}$, and the corresponding Euclidian norm and
Euclidian scalar product are denoted by $|\cdot|$ and $\langle\cdot
,\cdot\rangle$, respectively. For two vectors $a$ and $b$ in $\mathbb{R}^{k}$, we say that $a\geq b$ if
$a^{l}\geq b^{l}$, for each $1\leq l\leq k$.

\subsection{$G$-expectation}
Let $\Omega=C_{0}([0,\infty);\mathbb{R})$, the space of real-valued continuous
functions starting from the origin, be endowed with the distance
\[
\rho(\omega^{1},\omega^{2}):=\sum_{i=1}^{\infty}2^{-i}[(\max_{t\in
[0,i]}|\omega_{t}^{1}-\omega_{t}^{2}|)\wedge1], \text{ for } \omega^{1}%
,\omega^{2}\in\Omega.
\]
Let $B$ be the canonical process on $\Omega$. Set
\[
L_{ip}(\Omega):=\{\varphi(B_{t_{1}},...,B_{t_{n}}):\ n\in\mathbb{N}%
,\ t_{1},\cdots,t_{n}\in\lbrack0,\infty),\ \varphi\in C_{b,Lip}(\mathbb{R}%
^{n})\},
\]
where $C_{b,Lip}(\mathbb{R}^{n})$ denotes the set of bounded Lipschitz
functions on $\mathbb{R}^{n}$.  Let $G:\mathbb{R}\rightarrow
\mathbb{R}$ be defined by
\[
G(a)=\frac{1}{2}(\bar{\sigma}^{2}a^{+}-\underline{\sigma}^{2}a^{-}),
\]
for $0<\underline{\sigma}^{2}\leq\bar{\sigma}^{2}<\infty.$ The (conditional)
$G$-expectation for $\xi\in L_{ip}(\Omega)$ can be calculated as follows.
Assume
\[
\xi=\varphi(B_{{t_{1}}},B_{t_{2}},\cdots,B_{t_{n}}).
\]
Then, for $t\in\lbrack t_{k-1},t_{k})$, $k=1,\cdots,n$, we define
\[
\hat{\mathbb{E}}_{t}[\varphi(B_{{t_{1}}},B_{t_{2}},\cdots,B_{t_{n}}%
)]=u_{k}(t,B_{t};B_{t_{1}},\cdots,B_{t_{k-1}}),
\]
where, for any $k=1,\cdots,n$, $u_{k}(t,x;x_{1},\cdots,x_{k-1})$ is a function
of $(t,x)$ parameterized by $(x_{1},\cdots,x_{k-1})$ such that it solves the
following fully nonlinear PDE defined on $[t_{k-1},t_{k})\times\mathbb{R}$:
\[
\partial_{t}u_{k}+G(\partial_{x}^{2}u_{k})=0
\]
with terminal conditions
\[
u_{k}(t_{k},x;x_{1},\cdots,x_{k-1})=u_{k+1}(t_{k},x;x_{1},\cdots
,x_{k-1},x),\ k<n
\]
and $u_{n}(t_{n},x;x_{1},\cdots,x_{n-1})=\varphi(x_{1},\cdots,x_{n-1},x)$. The
$G$-expectation of $\xi$ is defined by  $\hat{\mathbb{E}}[\xi]:=\hat{\mathbb{E}}_{0}[\xi]$. We call $(\Omega,L_{ip}(\Omega),\hat{\mathbb{E}})$
 the $G$-expectation space.

For each $p\geq1$, the completion of $L_{ip} (\Omega)$ under the norm
$\Vert\xi\Vert_{L_{G}^{p}}:=(\hat{\mathbb{E}}[|\xi|^{p}])^{1/p}$ is denoted by
$L_{G}^{p}(\Omega)$. The conditional $G$-expectation $\mathbb{\hat{E}}%
_{t}[\cdot]$ can be extended continuously to the completion $L_{G}^{p}%
(\Omega)$. 
For each fixed $T\geq0$, set $\Omega_{T}=\{\omega_{\cdot\wedge T}:\omega
\in\Omega\}$. We may define $L_{ip}(\Omega_{T})$ and $L_{G}^{p}(\Omega_{T})$
similarly. Besides, Denis et al. \cite{DHP11} proved that the
$G$-expectation has the following representation.

\begin{theorem}
[\cite{DHP11}]\label{the1.1} There exists a weakly compact set $\mathcal{P}$
of probability measures on $(\Omega,\mathcal{B}(\Omega))$ such that
\[
\hat{\mathbb{E}}[\xi]=\sup_{P\in\mathcal{P}}E_{P}[\xi] \text{ for all } \xi
\in{L}_{G}^{1}{(\Omega)}.
\]
$\mathcal{P}$ is called a set that represents $\hat{\mathbb{E}}$.
\end{theorem}

Let $\mathcal{P}$ be a weakly compact set that represents $\hat{\mathbb{E}}$.
For this $\mathcal{P}$, we define the capacity
\[
c(A):=\sup_{P\in\mathcal{P}}P(A),\ A\in\mathcal{B}(\Omega).
\]
A set $A\in\mathcal{B}(\Omega)$ is called polar if $c(A)=0$. A property
holds $``quasi$-$surely"$ (q.s.) if it holds outside a polar set. In the
following, we do not distinguish two random variables $X$ and $Y$ if
$X=Y$, q.s.

For $\xi\in L_{ip}(\Omega_{T})$, let $\mathcal{E}(\xi)=\hat{\mathbb{E}}%
[\sup_{t\in[0,T]}\hat{\mathbb{E}}_{t}[\xi]]$ and $\mathcal{E}$ is called the
$G$-evaluation. For $p\geq1$ and $\xi\in L_{ip}(\Omega_{T})$, define
$\|\xi\|_{p,\mathcal{E}}=[\mathcal{E}(|\xi|^{p})]^{1/p}$ and denote by
$L_{\mathcal{E}}^{p}(\Omega_{T})$ the completion of $L_{ip}(\Omega_{T})$ under
$\|\cdot\|_{p,\mathcal{E}}$. The following theorem can be regarded as Doob's
maximal inequality under $G$-expectation.

\begin{theorem}
[\cite{S11}]\label{the1.2} For any $\alpha\geq1$ and $\delta>0$,
$L_{G}^{\alpha+\delta}(\Omega_{T})\subset L_{\mathcal{E}}^{\alpha}(\Omega
_{T})$. More precisely, for any $1<\gamma<\beta:=(\alpha+\delta)/\alpha$,
$\gamma\leq2$, we have
\[
\|\xi\|_{\alpha,\mathcal{E}}^{\alpha}\leq\gamma^{*}\{\|\xi\|_{L_{G}%
^{\alpha+\delta}}^{\alpha}+14^{1/\gamma} C_{\beta/\gamma}\|\xi\|_{L_{G}%
^{\alpha+\delta}}^{(\alpha+\delta)/\gamma}\},\quad\forall\xi\in L_{ip}%
(\Omega_{T}),
\]
where $C_{\beta/\gamma}=\sum_{i=1}^{\infty}i^{-\beta/\gamma}$, $\gamma
^{*}=\gamma/(\gamma-1)$.
\end{theorem}

For $T>0$ and $p\geq1$, the following spaces will be frequently used in this paper.

\begin{itemize}
\item $M_{G}^{0}(0,T):=\{\eta: \eta_{t}(\omega)=\sum_{j=0}^{N-1}\xi_{j}%
(\omega)\mathbf{1}_{[t_{j},t_{j+1})}(t),$ where $\xi_{j}\in L_{ip}%
(\Omega_{t_{j}})$, $t_{0}\leq\cdots\leq t_{N}$ is a partition of $[0,T]\}$;

\item $M_{G}^{p}(0,T)$ is the completion of $M_{G}^{0}(0,T)$ under the norm
$\Vert\eta\Vert_{M_{G}^{p}}:=(\mathbb{\hat{E}}[\int_{0}^{T}|\eta_{s}%
|^{p}ds])^{1/p}$;

\item $H_{G}^{p}(0,T)$ is the completion of $M_{G}^{0}(0,T)$ under the
norm $\|\eta\|_{H_{G}^{p}}:=\{\hat{\mathbb{E}}[(\int_{0}^{T}|\eta_{s}%
|^{2}ds)^{p/2}]\}^{1/p}$;

\item $S_{G}^{0}(0,T)=\{h(t,B_{t_{1}\wedge t}, \ldots,B_{t_{n}\wedge t}%
):t_{1},\ldots,t_{n}\in[0,T],h\in C_{b,Lip}(\mathbb{R}^{n+1})\}$;

\item $S_{G}^{p}(0,T)$ is the completion of $S_{G}^{0}(0,T)$ under the norm
$\Vert\eta\Vert_{S_{G}^{p}}=\{\hat{\mathbb{E}}[\sup_{t\in\lbrack0,T]}|\eta
_{t}|^{p}]\}^{1/p}$;

\item $\mathcal{A}_{G}^{p}(0,T)$ is the collection of processes $K\in S_{G}%
^{p}(0,T)$ such that $K$ is a non-increasing $G$-martingale with $K_{0}=0$.
\end{itemize}

We denote by $\langle B\rangle$ the quadratic variation process of the
$G$-Brownian motion $B$. For two processes $\eta\in M_{G}^{p}(0,T)$ and
$\zeta\in H_{G}^{p}(0,T)$, Peng \cite{P10} and Li and Peng \cite{LP} established the $G$-It\^{o} integrals
$\int_{0}^{\cdot}\eta_{s}d\langle B\rangle_{s}$ and $\int_{0}^{\cdot}\zeta
_{s}dB_{s}$.

\subsection{Multi-dimensional $G$-BSDEs}
 We denote by $M_{G}%
^{p}(0,T;\mathbb{R}^{k})$ the set of $k$-dimensional stochastic process
$X=(X^{1},\cdots,X^{k})$ 
such that $X^{l}\in M_{G}^{p}(0,T)$, $1\leq l\leq k$, and
we also define ${S}_{G}^{\alpha}(0,T;\mathbb{R}^{k})$, $H_{G}^{\alpha
}(0,T;\mathbb{R}^{k})$, $\mathcal{A}_{G}^{\alpha}(0,T;\mathbb{R}^{k})$
and $L_{G}^{\beta}(\Omega_{T};\mathbb{R}^{k})$ similarly. Consider the
following type of $k$-dimensional $G$-BSDE with diagonal generators on the
interval $[0,T]$:
\begin{equation}
Y_{t}^{l}=\xi^{l}+\int_{t}^{T}f^{l}(s,Y_{s},Z_{s}^{l})ds+\int_{t}^{T}
g^{l}(s,Y_{s},Z_{s}^{l})d\langle B\rangle_{s}-\int_{t}^{T} Z_{s}^{l}%
dB_{s}-(K_{T}^{l}-K_{t}^{l}),\ 1\leq l\leq k. \label{my1}%
\end{equation}
Here by \textit{diagonal} we mean that in the generators%
\[
f^{l}(t,\omega,y,z^{l}),g^{l}(t,\omega,y,z^{l}):[0,T]\times\Omega_{T}%
\times\mathbb{R}^{k}\times\mathbb{R}\rightarrow\mathbb{R},\ \ \forall1\leq
l\leq k,
\]
the $z$ parts of the $l$-th components $f^{l},g^{l}$ only depend on $z^{l}$.
We impose that

\begin{itemize}
\item[(A1)] there is some constant $\beta>2$ such that for each $y\in \mathbb{R}^k,z\in\mathbb{R}$,
$f^{l}(\cdot,\cdot,y,z),g^{l}(\cdot,\cdot,y,z)\in M_{G}^{\beta}(0,T)$, $1\leq l\leq k$;

\item[(A2)] there exists some $L>0$ such that, for each $1\leq l\leq k$, $y_{1},y_{2}%
\in\mathbb{R}^{k},z_{1},z_{2}\in\mathbb{R}$,
\[
|f^{l}(t,y_{1},z_{1})-f^{l}(t,y_{2},z_{2})|+%
|g^{l}(t,y_{1},z_{1})-g^{l}(t,y_{2},z_{2})|\leq L(|y_{1}-y_{2}%
|+|z_{1}-z_{2}|).
\]

\end{itemize}

\begin{theorem}[\cite{Liu-stochastics}]
\label{my17} Suppose $\xi \in L_G^\beta(\Omega_T;\mathbb{R}^k)$ and (A1)-(A2) are satisfied for some $\beta>2$. Then for
any $2\leq\alpha<\beta$, the $G$-BSDE (\ref{my1}) has a unique solution
$(Y,Z,K)\in{S}_{G}^{\alpha}(0,T;\mathbb{R}^{k})\times H_{G}^{\alpha
}(0,T;\mathbb{R}^{k})\times\mathcal{A}_{G}^{\alpha}(0,T;\mathbb{R}^{k})$.
Moreover, $Y\in{M}_{G}^{\beta}(0,T;\mathbb{R}^{k}).$
\end{theorem}

Then, we present the comparison theorem for multi-dimensional $G$-BSDEs.
\begin{theorem}[\cite{Liu-stochastics}]
\label{Myth2-3} Given two $G$-BSDEs on the interval $[0,T]$:
\[%
  Y_{t}^{l}=\xi^{l}+\int_{t}^{T}f^{l}(s,Y_{s},Z_{s}^{l})ds+\int_{t}^{T}%
g^{l}(s,Y_{s},Z_{s}^{l})d\langle B\rangle_{s}+V_{T}^{l}-V_{t}^{l}-\int_{t}%
^{T}Z_{s}^{l}dB_{s}-(K_{T}^{l}-K_{t}^{l}),
\]
and
\[%
 \bar{Y}_{t}^{l}=\bar{\xi}^{l}+\int_{t}^{T}\bar{f}^{l}(s,\bar{Y}_{s},\bar
{Z}_{s}^{l})ds+\int_{t}^{T}\bar{g}^{l}(s,\bar{Y}_{s},\bar{Z}_{s}^{l})d\langle
B\rangle_{s}+\bar{V}_{T}^{l}-\bar{V}_{T}^{l}-\int_{t}^{T}\bar{Z}_{s}^{l}%
dB_{s}-(\bar{K}_{T}^{l}-\bar{K}_{t}^{l}),
\]
where $1\leq l\leq k$. For any $1\leq l\leq k$, suppose that $f^{l}(t,y,z^{l}),\bar{f}^{l}(t,\bar{y},z^{l}),g^{l}%
(t,y,z^{l}),\bar{g}^{l}(t,\bar{y},z^{l})$ satisfy (A1)-(A2), $\xi^{l}%
,\bar{\xi}^{l}\in L_{G}^{\beta}(\Omega_{T})$ and $V_{t}^{l},\bar{V}_{t}^{l}$
are RCLL (right-continuous with left limits) such that $\hat{\mathbb{E}}%
[\sup_{t\in\lbrack0,T]}|V_{t}^{l}|^{\beta}]\vee\hat{\mathbb{E}}[\sup_{t\in
\lbrack0,T]}|V_{t}^{l}|^{\beta}]<\infty,$ for some $\beta>2.$ Assume the
following conditions hold:

\begin{itemize}
\item[(i)] for any $z^{l}%
\in\mathbb{R}$ and $y,\bar{y}\in\mathbb{R}^{k}$ satisfying $y^{j}\geq\bar
{y}^{j}$ for $j\neq l$ and $y^{l}=\bar{y}^{l}$, it holds that $f^{l}%
(t,y,z^{l})\geq\bar{f}^{l}(t,\bar{y},z^{l}),g^{l}(t,y,z^{l})\geq\bar{g}%
^{l}(t,\bar{y},z^{l})$, for $1\leq l\leq k$;

\item[(ii)] $\xi\geq\bar{\xi}$ and $V_{t}^{l}-\bar{V}_{t}^{l}$ is non-decreasing, $1\leq l\leq k$.
\end{itemize}
Then $Y_{t}\geq\bar{Y}_{t}$ for each $t\in\lbrack0,T]$, q.s.

\end{theorem}

\subsection{One-dimensional reflected $G$-BSDEs}
 Consider the
following parameters: the generators $f$ and $g$, the obstacle process
$\{S_{t}\}_{t\in\lbrack0,T]}$ and the terminal value $\xi$. We assume that the
generators
\[
f(t,\omega,y,z),g(t,\omega,y,z):[0,T]\times\Omega_{T}\times\mathbb{R}%
^{2}\rightarrow\mathbb{R},
\]
satisfy:

\begin{itemize}
\item[(H1)] for any $y,z$, $f(\cdot,\cdot,y,z)$, $g(\cdot,\cdot,y,z)\in
M_{G}^{\beta}(0,T)$ with $\beta>2$

\item[(H2)] $|f(t,\omega,y,z)-f(t,\omega,y^{\prime},z^{\prime})|+|g(t,\omega
,y,z)-g(t,\omega,y^{\prime},z^{\prime})|\leq L(|y-y^{\prime}|+|z-z^{\prime}|)$
for some $L>0$.
\end{itemize}

The lower obstacle $S$ is bounded from above by some given generalized $G$-It\^{o}
process. More precisely, it satisfies the following condition:

\begin{itemize}
\item[(H3)] $S_{t}\leq I_{t}$ for any $t\in[0,T]$, where $I$ is a generalized
$G$-It\^{o} process:
\[
I_{t}=I_{0}+\int_{0}^{t} b^{I}(s)ds+\int_{0}^{t} \sigma^{I}(s)dB_{s}+K^{I}%
_{t},
\]
with $I_0\in\mathbb{R}$, $b^{I}\in M_{G}^{\beta}(0,T)$, $\sigma^{I}\in H_{G}^{\beta}(0,T)$,
$K^I\in \mathcal{A}_G^\beta(0,T)$.
\end{itemize}

The terminal value $\xi$ satisfies

\begin{itemize}
\item[(H4)] $\xi\in L_{G}^{\beta}(\Omega_{T})$ and $\xi\geq S_{T}$, $q.s.$
\end{itemize}

A triple of processes $(Y,Z,A)$ is called a solution of reflected $G$-BSDE
with a lower obstacle $S$, terminal value $\xi$ and generators $f,g$, if for
some $2\leq\alpha\leq\beta$ the following properties hold:

\begin{itemize}
\item[(a)] $(Y,Z,A)\in\mathcal{S}_{G}^{\alpha}(0,T)$ and $Y_{t}\geq S_{t}$,
$0\leq t\leq T$;

\item[(b)] $Y_{t}=\xi+\int_{t}^{T} f(s,Y_{s},Z_{s})ds+\int_{t}^{T}
g(s,Y_{s},Z_{s})d\langle B\rangle_{s} -\int_{t}^{T} Z_{s} dB_{s}+(A_{T}%
-A_{t})$;

\item[(c)] $\{-\int_{0}^{t} (Y_{s}-S_{s})dA_{s}\}_{t\in[0,T]}$ is a
non-increasing $G$-martingale.
\end{itemize}
Here, we denote by $\mathcal{S}_{G}^{\alpha}(0,T)$ the collection of processes
$(Y,Z,A)$ such that $Y\in S_{G}^{\alpha}(0,T)$, $Z\in H_{G}^{\alpha}(0,T)$,
$A$ is a continuous non-decreasing process with $A_{0}=0$ and $A\in
S_{G}^{\alpha}(0,T)$.

\begin{theorem}[\cite{LPSH,LS}]
\label{the1.14} Suppose that $(\xi,f,g,S)$ satisfy \textsc{(H1)}%
-\textsc{(H4)}. Then, the reflected $G$-BSDE with parameters $(\xi,f,g,S)$
has a unique solution $(Y,Z,A)\in\mathcal{S}_{G}^{\alpha}(0,T)$, for any
$2\leq\alpha<\beta$. Besides, there exists a constant $C:=C(\alpha,T,
L,\underline{\sigma})>0$ such that {for $2\leq\alpha\leq\beta$,} 
\[
|Y_{t}|^{\alpha}\leq C\hat{\mathbb{E}}_{t}[|\xi|^{\alpha}+\int_{t}%
^{T}(|f(s,0,0)|^{\alpha}+|g(s,0,0)|^{\alpha}+|b^{I}(s)|^{\alpha}+|\sigma
^{I}(s)|^{\alpha})ds+\sup_{s\in[t,T]}|I_{s}|^{\alpha}].
\]
\end{theorem}

Then, we present some a priori estimates. It is worth pointing out that in the following two propositions, we do not need to assume that $(Y,Z,A)$ and $(Y^{(i)},Z^{(i)},A^{(i)})$, $i=1,2$, are the solutions of reflected $G$-BSDEs, i.e., the condition (c) is not needed.
\begin{proposition}[\cite{LPSH}]
\label{the1.6} Let $f,g$ satisfy \textsc{(H1)} and \textsc{(H2)}.
Assume
\[
Y_{t}=\xi+\int_{t}^{T} f(s,Y_{s},Z_{s})ds+\int_{t}^{T} g(s,Y_{s}%
,Z_{s})d\langle B\rangle_{s}-\int_{t}^{T} Z_{s}dB_{s}+(A_{T}-A_{t}),
\]
where $(Y,Z,A)\in\mathcal{S}_{G}^{\alpha}(0,T)$ with $2\leq \alpha\leq \beta$. Then,
there exists a constant $C:=C(\alpha, T, L,\underline{\sigma})>0$ such that
for each $t\in[0,T]$,
\begin{align*}
\hat{\mathbb{E}}_{t}[(\int_{t}^{T} |Z_{s}|^{2}ds)^{\frac{\alpha}{2}}]  &  \leq
C\{\hat{\mathbb{E}}_{t}[\sup_{s\in[t,T]}|Y_{s}|^{\alpha}] +(\hat{\mathbb{E}%
}_{t}[\sup_{s\in[t,T]}|Y_{s}|^{\alpha}])^{1/2}(\hat{\mathbb{E}}_{t}[(\int%
_{t}^{T} h_{s} ds)^{\alpha}])^{1/2}\},\\
\hat{\mathbb{E}}_{t}[|A_{T}-A_{t}|^{\alpha}]  &  \leq C\{\hat{\mathbb{E}}%
_{t}[\sup_{s\in[t,T]}|Y_{s}|^{\alpha}]+\hat{\mathbb{E}}_{t}[(\int_{t}^{T}
h_{s} ds)^{\alpha}]\},
\end{align*}
where $h_{s}=|f(s,0,0)|+|g(s,0,0)|$.
\end{proposition}

\begin{proposition}[\cite{LPSH}]
\label{the1.7} For $i=1,2$, let $\xi^{(i)}\in L_{G}^{\beta
}(\Omega_{T})$, $f^{(i)},g^{(i)}$ satisfy \textsc{(H1)} and \textsc{(H2)} for
some $\beta>2$. Assume
\[
Y_{t}^{(i)}=\xi^{(i)}+\int_{t}^{T}f^{(i)}(s,Y_{s}^{(i)},Z_{s}^{(i)}%
)ds+\int_{t}^{T}g^{(i)}(s,Y_{s}^{(i)},Z_{s}^{(i)})d\langle B\rangle_{s}%
-\int_{t}^{T}Z_{s}^{(i)}dB_{s}+(A_{T}^{(i)}-A_{t}^{(i)}),
\]
where $(Y^{(i)},Z^{(i)},A^{(i)})\in\mathcal{S}_{G}^{\alpha}(0,T)$ for some
$2\leq\alpha\leq\beta$. Set $\hat{Y}_{t}=Y_{t}^{(1)}-Y_{t}^{(2)}$, $\hat
{Z}_{t}=Z_{t}^{(1)}-Z_{t}^{(2)}$. Then, there exists a constant $C:=C(\alpha
,T,L,\underline{\sigma})$ such that
\begin{align*}
\hat{\mathbb{E}}[(\int_{0}^{T}|\hat{Z}|^{2}ds)^{\frac{\alpha}{2}}]\leq &
C_{\alpha}\{(\hat{\mathbb{E}}[\sup_{t\in\lbrack0,T]}|\hat{Y}_{t}|^{\alpha
}])^{1/2}\sum_{i=1}^{2}[(\hat{\mathbb{E}}[\sup_{t\in\lbrack0,T]}|{Y}_{t}%
^{(i)}|^{\alpha}])^{1/2}\\
&  +(\hat{\mathbb{E}}[(\int_{0}^{T}h_{s}^{(i)}ds)^{\alpha}])^{1/2}%
]+\hat{\mathbb{E}}[\sup_{t\in\lbrack0,T]}|\hat{Y}_{t}|^{\alpha}]\},
\end{align*}
where $h_{s}^{(i)}=|f^{(i)}(s,0,0)|+|g^{(i)}(s,0,0)|$.
\end{proposition}



\begin{proposition}[\cite{LPSH,LS}]
\label{the1.10} Let $(\xi^{(1)},f^{(1)},g^{(1)},S^{(1)})$ and $(\xi
^{(2)},f^{(2)},g^{(2)},S^{(2)})$ be two sets of data satisfying  \textsc{(H1)}--\textsc{(H4)}. Let $(Y^{(i)},Z^{(i)},A^{(i)}%
)\in\mathcal{S}_{G}^{\alpha}(0,T)$ be the solutions of the reflected $G$-BSDEs
with data $(\xi^{(i)},f^{(i)},g^{(i)},S^{(i)})$, $i=1,2$ respectively, with
$2\leq\alpha\leq \beta$. Set $\hat{Y}_{t}=Y_{t}^{(1)}-Y_{t}^{(2)}$, $\hat
{S}_{t}=S_{t}^{(1)}-S_{t}^{(2)},$ $\hat{\xi}=\xi^{(1)}-\xi^{(2)}$. Then, there
exists a constant $C:=C(\alpha,T,L,\underline{\sigma})>0$ such that
\[
|\hat{Y}_{t}|^{\alpha}\leq C\{\hat{\mathbb{E}}_{t}[|\hat{\xi}|^{\alpha}%
+\int_{t}^{T}|\hat{h}|^\alpha_{s}ds]+(\hat{\mathbb{E}}_{t}[\sup_{s\in\lbrack
t,T]}|\hat{S}_{s}|^{\alpha}])^{\frac{1}{\alpha}}\Psi_{t,T}^{\frac{\alpha
-1}{\alpha}}\},
\]
where $\hat{h}_{s}=|f^{(1)}(s,Y_{s}^{(2)},Z_{s}^{(2)})-f^{(2)}(s,Y_{s}%
^{(2)},Z_{s}^{(2)})|+|g^{(1)}(s,Y_{s}^{(2)},Z_{s}^{(2)})-g^{(2)}(s,Y_{s}%
^{(2)},Z_{s}^{(2)})|$ and
\begin{align*}
\Psi_{t,T}=\sum_{i=1}^{2}&\hat{\mathbb{E}}_{t}\bigg[\sup_{s\in\lbrack t,T]}%
\hat{\mathbb{E}}_{s}\big[|\xi^{(i)}|^{\alpha}+\sup_{s\in\lbrack0,T]}|I_{s}%
^{(i)}|^{\alpha}\\
&+\int_{0}^{T}(|f^{(i)}(s,0,0)|^{\alpha}+|g^{(i)}%
(s,0,0)|^{\alpha}+|b^{(i)}(s)|^{\alpha}+|\sigma^{(i)}(s)|^{\alpha})ds\big]\bigg].
\end{align*}

\end{proposition}

The following result is the comparison theorem for reflected $G$-BSDEs. 
\begin{theorem}[\cite{LPSH,LS}]
\label{the1.16} Let $(\xi^{(1)},f^{(1)},g^{(1)},S^{(1)})$ and $(\xi
^{(2)},f^{(2)},g^{(2)},S^{(2)})$ be two sets of parameters satisfying
\textsc{(H1)}--\textsc{(H4)}. We furthermore assume the following:

\begin{itemize}
\item[(i)] $\xi^{(1)}\geq\xi^{(2)}$, $q.s.$;

\item[(ii)] $f^{(1)}(t,y,z)\geq f^{(2)}(t,y,z)$, $g^{(1)}(t,y,z)\geq
g^{(2)}(t,y,z)$, $\forall(y,z)\in\mathbb{R}^{2}$;

\item[(iii)] $S_{t}^{(1)}\geq S_{t}^{(2)}$, $0\leq t\leq T$, $q.s.$.
\end{itemize}
Let $(Y^{(i)},Z^{(i)},A^{(i)})$ be the solutions of the reflected $G$-BSDE
with parameters $(\xi^{(i)},f^{(i)},g^{(i)},S^{(i)})$, $i=1,2$, respectively.
Then, we have
$
Y_{t}^{(1)}\geq Y_{t}^{(2)}$, $0\leq t\leq T$, q.s.
\end{theorem}

\section{Multi-dimensional reflected $G$-BSDEs}

In this paper, we shall consider the $k$-dimensional reflected $G$-BSDE with
diagonal generators in the following form:
\begin{equation}
Y_{t}^{l}=\xi^{l}+\int_{t}^{T}f^{l}(s,Y_{s},Z_{s}^{l})ds+\int_{t}^{T}%
g^{l}(s,Y_{s},Z_{s}^{l})d\langle B\rangle_{s}-\int_{t}^{T}Z_{s}^{l}%
dB_{s}+(A_{T}^{l}-A_{t}^{l}),\ 1\leq l\leq k, \label{Myeq3-1}%
\end{equation}
where $f^l,g^l$ satisfy (A1) and (A2) as in Subsection 2.2, $1\leq l\leq k$. We call $(Y,Z,A)=((Y^{l},1\leq l\leq k)^{T},(Z^{l},1\leq l\leq k)^{T}%
,(A^{l},1\leq l\leq k)^{T})$ a solution for  the reflected $G$-BSDE with generators $f,g$, terminal value
$\xi$ and lower obstacle $S$, if it satisfies the following conditions:

\begin{itemize}
\item[(1)] $(Y^{l},Z^{l},A^{l})\in\mathcal{S}_{G}^{\alpha}(0,T)$ and
$Y_{t}^{l}\geq S_{t}^{l}$, $t\in[0,T]$, where $2\leq\alpha\leq\beta$, $1\leq
l\leq k$ ($\beta$ is the order of integrability for the parameters);

\item[(2)] $Y_{t}^{l}=\xi^{l}+\int_{t}^{T} f^{l}(s,Y_{s},Z_{s}^{l})ds+\int%
_{t}^{T} g^{l}(s,Y_{s},Z_{s}^{l})d\langle B\rangle_{s}-\int_{t}^{T} Z_{s}^{l}
dB_{s}+A_{T}^{l}-A_{t}^{l}$, $1\leq l\leq k$;

\item[(3)] $\{-\int_{0}^{t} (Y_{s}^{l}-S_{s}^{l})dA_{s}^{l}\}_{t\in[0,T]}$ is
a non-increasing $G$-martingale, $1\leq l\leq k$.
\end{itemize}

We need to propose the following assumptions on the obstacle $S$ and terminal value $\xi$: there exists
some $\beta>2$ such that 

\begin{itemize}


\item[(A3)] for any $1\leq l\leq k$, $S^{l}_{t}\leq I^{l}_{t}$, $t\in[0,T]$,
where $I^{l}$ is a generalized $G$-It\^{o} process:
\[
I^{l}_{t}=I^{l}_{0}+\int_{0}^{t} b^{l}(s)ds+\int_{0}^{t} \sigma^{l}%
(s)dB_{s}+K^{l}_{t},
\]
with $I_0\in\mathbb{R}$, $b^{l}\in M_{G}^{\beta}(0,T)$, $\sigma^{l}\in H_{G}^{\beta}(0,T)$, $K^l\in \mathcal{A}_G^\beta(0,T)$;

\item[(A4)] for any $1\leq l\leq k$, $S_{T}^{l}\leq\xi^{l}$ and $\xi^{l}\in
L_{G}^{\beta}(\Omega_{T})$.
\end{itemize}

For simplicity,  we denote
\[
\phi(t,y,z)=(\phi^{1}(t,y,z^{1}),\cdots,\phi^{k}(t,y,z^{k}))^{T},
\]
for any $t\in\lbrack0,T]$, $y=(y^{1}%
,\cdots,y^{k})^{T},z=(z^{1},\cdots,z^{k})^{T}\in\mathbb{R}^{k}$ and any
function $\phi^{l}:[0,T]\times\Omega_{T}\times\mathbb{R}^{k}\times
\mathbb{R}\rightarrow\mathbb{R}$, $1\leq l\leq k$.
 We denote by $\mathcal{S}_{G}^{\alpha}(0,T;\mathbb{R}^{k})$ the collection
of $k$-dimensional triples $(Y,Z,A)=((Y^{1},Z^{1},A^{1}),\cdots,(Y^{k}%
,Z^{k},A^{k}))^{T},$ for $(Y^{l},Z^{l},A^{l})\in\mathcal{S}_{G}^{\alpha}(0,T),
1\leq l\leq k$. 

We first introduce the main result in this paper:

\begin{theorem}
\label{MainThm} Assume that $(\xi,f,g,S)$ satisfy Assumptions
(A1)-(A4). Then,  the multi-dimensional reflected $G$-BSDE (\ref{Myeq3-1}) has a
unique solution $(Y,Z,A)\in\mathcal{S}_{G}^{\alpha}(0,T;\mathbb{R}^{k})$.
\end{theorem}

\subsection{A priori estimates}

In this subsection, we present some useful a priori estimates for multi-dimensional reflected $G$-BSDEs. For simplicity, we only consider the case that   $g^l\equiv0,l=1,\cdots,k,$ but similar results still hold for the general case. In the sequel, $C$ will always be a universal constant which may change from line to line and $\mathbf{0}$ represents the $k$-dimensional zero vector.


\begin{proposition}
\label{estimate for Y} Let $(Y,Z,A)$ be the solution of reflected $G$-BSDE
with parameters $(\xi,f,S)$ satisfying (A1)-(A4). Then, for any $2\leq
\alpha\leq \beta$, there exists a constant $C$ depending on $T,G,k,L$ and $\alpha$
such that
\[
|Y_{t}|^{\alpha}\leq C\hat{\mathbb{E}}_{t}[|\xi|^{\alpha}+\sup_{s\in
[t,T]}|I_{s}|^{\alpha}+\int_{t}^{T} (|f(s,\mathbf{0},0)|^{\alpha
}+|b(s)|^{\alpha}+|\sigma(s)|^{\alpha})ds].
\]

\end{proposition}

\begin{proof}
By Theorem \ref{the1.14}, for any $1\leq l\leq k$ and $t\leq r\leq T$, we
have
\begin{align*}
|Y_{r}^{l}|^{\alpha}\leq &  C\hat{\mathbb{E}}_{r}[|\xi^{l}|^{\alpha}%
+\sup_{s\in[r,T]}|I^{l}_{s}|^{\alpha}+\int_{r}^{T} (|\tilde{f}^{l}%
(s)|^{\alpha}+|b^{l}(s)|^{\alpha}+|\sigma^{l}(s)|^{\alpha})ds]\\
\leq &  C\hat{\mathbb{E}}_{r}[|\xi|^{\alpha}+\sup_{s\in[t,T]}|I_{s}|^{\alpha
}+\int_{t}^{T} (|f(s,\mathbf{0},0)|^{\alpha}+|b(s)|^{\alpha}+|\sigma
(s)|^{\alpha})ds+\int_{r}^{T} |Y_{s}|^{\alpha}ds],
\end{align*}
where $\tilde{f}^{l}(s)=f^{l}(s,Y^{1}_{s},\cdots,Y^{l-1}_{s},0,Y^{l+1}%
_{s},\cdots,Y^{k}_{s},0)$. Summing up over $l$, we obtain that
\[
|Y_{r}|^{\alpha}\leq C\hat{\mathbb{E}}_{r}[|\xi|^{\alpha}+\sup_{s\in
[t,T]}|I_{s}|^{\alpha}+\int_{t}^{T} (|f(s,\mathbf{0},0)|^{\alpha
}+|b(s)|^{\alpha}+|\sigma(s)|^{\alpha})ds]+C\int_{r}^{T} \hat{\mathbb{E}}%
_{r}[|Y_{s}|^{\alpha}] ds.
\]
Taking conditional expectations on both sides of the above equation, we have
\[
\hat{\mathbb{E}}_{t}[|Y_{r}|^{\alpha}] \leq C\hat{\mathbb{E}}_{t}%
[|\xi|^{\alpha}+\sup_{s\in[t,T]}|I_{s}|^{\alpha}+\int_{t}^{T} (|f(s,\mathbf{0}%
,0)|^{\alpha}+|b(s)|^{\alpha}+|\sigma(s)|^{\alpha})ds]+C\int_{r}^{T}
\hat{\mathbb{E}}_{t}[|Y_{s}|^{\alpha}] ds.
\]
It follows from the Gronwall inequality that
\[
\hat{\mathbb{E}}_{t}[|Y_{r}|^{\alpha}] \leq C\hat{\mathbb{E}}_{t}%
[|\xi|^{\alpha}+\sup_{s\in[t,T]}|I_{s}|^{\alpha}+\int_{t}^{T} (|f(s,\mathbf{0}%
,0)|^{\alpha}+|b(s)|^{\alpha}+|\sigma(s)|^{\alpha})ds],
\]
which is the desired result by letting $r=t$.
\end{proof}

\begin{proposition}
\label{difference of Y}  Assume that $\prescript{i}{}{\xi}^{l}$,
$\prescript{i}{}{f}^{l}$, $\prescript{i}{}{S}^{l}$ satisfy (A1)-(A4), 
	$1\leq l\leq k$, $i=1,2$. Let
$(\prescript{i}{}{Y},\prescript{i}{}{Z},\prescript{i}{}{A})\in\mathcal{S}%
_{G}^{\alpha}(0,T;\mathbb{R}^{k})$ be the solution to reflected $G$-BSDEs with
parameters
$(\prescript{i}{}{\xi},\prescript{i}{}{f},\prescript{i}{}{S})$, $i=1,2$,
for some $2\leq\alpha\leq \beta$. Set $\hat{Y}_{t}=\prescript{1}{}{Y}_{t}%
-\prescript{2}{}{Y}_{t}$. Then there exists a constant $C$ depending on
$T,G,k,L$ and $\alpha$ such that
\[
|\hat{Y}_{t}|^{\alpha}\leq C\{\hat{\mathbb{E}}_{t}[|\hat{\xi}|^{\alpha}%
+\int_{t}^{T}|\hat{f}_{s}|^{\alpha}ds]+(\hat
{\mathbb{E}}_{t}[\sup_{s\in\lbrack t,T]}|\hat{S}_{s}|^{\alpha}])^{\frac
{1}{\alpha}}\Psi_{t,T}^{\frac{\alpha-1}{\alpha}}\},
\]
where $\hat{\xi}=\prescript{1}{}{\xi}-\prescript{2}{}{\xi}$, $\hat{f}%
_{s}=\prescript{1}{}{f}(s,\prescript{2}{}{Y}_{s},\prescript{2}{}{Z}_{s}%
)-\prescript{2}{}{f}(s,\prescript{2}{}{Y}_{s},\prescript{2}{}{Z}_{s})$,
$\hat{S}_{t}=\prescript{1}{}{S}_{t}-\prescript{2}{}{S}_{t}$ and
\[
\Psi_{t,T}=\sum_{i=1}^{2}\hat{\mathbb{E}}_{t}[|\prescript{i}{}{\xi}|^{\alpha
}+\sup_{s\in\lbrack t,T]}|\prescript{i}{}{I}_{s}|^{\alpha}+\int_{t}%
^{T}|\prescript{i}{}{f}(s,\mathbf{0},0)|^{\alpha}%
+|\prescript{i}{}{b}(s)|^{\alpha}+|\prescript{i}{}{\sigma}(s)|^{\alpha
}+|\prescript{i}{}{Y}_{s}|^{\alpha}ds].
\]

\end{proposition}

\begin{proof}
For any $1\leq l\leq k$, by Proposition \ref{the1.10}, we have for any
$s\in[t,T]$, 
\begin{align*}
|\hat{Y}^{l}_{s}|^{\alpha}  &  \leq C\{\hat{\mathbb{E}}_{s}[|\hat{\xi}%
^{l}|^{\alpha}+\int_{s}^{T} |\hat{f}^l_r|^{\alpha}dr]+(\hat{\mathbb{E}%
}_{s}[\sup_{u\in[s,T]}|\hat{S}^{l}_{u}|^{\alpha}])^{\frac{1}{\alpha}}(\Psi
^{l}_{s,T})^{\frac{\alpha-1}{\alpha}}\}\\
&  \leq C\{\hat{\mathbb{E}}_{s}[|\hat{\xi}|^{\alpha}+\int_{s}^{T} |\hat{f}%
_{r}|^{\alpha}dr +\int_{s}^{T} |\hat{Y}_{r}|^{\alpha}dr]+(\hat{\mathbb{E}}%
_{s}[\sup_{u\in[s,T]}|\hat{S}_{u}|^{\alpha}])^{\frac{1}{\alpha}}\Psi
_{s,T}^{\frac{\alpha-1}{\alpha}}\},
\end{align*}
where $\hat{f}^l_r=\prescript{1}{}{f}^{l}(r,
	\prescript{1}{}{Y}^{(l)}_{r},\prescript{2}{}{Z}^l_{r})-\prescript{2}{}{f}^{l}(r,
	\prescript{2}{}{Y}_{r},\prescript{2}{}{Z}^l_{r})$, $\prescript{1}{}{Y}^{(l)}=(\prescript{1}{}{Y}^{1}, \cdots,
\prescript{1}{}{Y}^{l-1},\prescript{2}{}{Y}^{l},\prescript{1}{}{Y}^{l+1}%
,\cdots,\prescript{1}{}{Y}^{k})$ and
\begin{align*}
\Psi^{l}_{s,T}=\sum_{i=1}^{2}\hat{\mathbb{E}}_{s}[|\prescript{i}{}{\xi}^{l}%
|^{\alpha}+\sup_{u\in[s,T]}|\prescript{i}{}{I}^{l}_{u}|^{\alpha}+\int_{s}^{T}
(|\prescript{i}{}{f}^{l}(r)|^{\alpha}+|\prescript{i}{}{b}^{l}(r)|^{\alpha
}+|\prescript{i}{}{\sigma}^{l}(r)|^{\alpha})dr],
\end{align*}
and $\prescript{i}{}{f}^{l}(r)=f^{l}(r,\prescript{i}{}{Y}^{1}_{r}%
,\cdots,\prescript{i}{}{Y}^{l-1}_{r},0,\prescript{i}{}{Y}^{l+1}_{r}%
,\cdots,\prescript{i}{}{Y}^{k}_{r},0)$. Summing up over $l$, we obtain that
for any $s\geq t$,
\[
|\hat{Y}_{s}|^{\alpha}\leq C\{\hat{\mathbb{E}}_{s}[|\hat{\xi}|^{\alpha}%
+\int_{t}^{T} |\hat{f}_{r}|^{\alpha}dr]+(\hat{\mathbb{E}}_{s}[\sup_{u\in
[t,T]}|\hat{S}_{u}|^{\alpha}])^{\frac{1}{\alpha}}\Psi_{s,T}^{\frac{\alpha
-1}{\alpha}}\} +C\int_{s}^{T} \hat{\mathbb{E}}_{s}[|\hat{Y}_{r}|^{\alpha}] dr.
\]
Taking conditional expectations on both sides implies that
\[
\hat{\mathbb{E}}_{t}[|\hat{Y}_{s}|^{\alpha}]\leq C\{\hat{\mathbb{E}}_{t}%
[|\hat{\xi}|^{\alpha}+\int_{t}^{T} |\hat{f}_{r}|^{\alpha}dr]+(\hat{\mathbb{E}%
}_{t}[\sup_{u\in[t,T]}|\hat{S}_{u}|^{\alpha}])^{\frac{1}{\alpha}}\Psi
_{t,T}^{\frac{\alpha-1}{\alpha}}\} +C\int_{s}^{T} \hat{\mathbb{E}}_{t}[|\hat
{Y}_{r}|^{\alpha}] dr.
\]
Applying the Gronwall inequality, we have for any $s\geq t$
\[
\hat{\mathbb{E}}_{t}[|\hat{Y}_{s}|^{\alpha}]\leq C\{\hat{\mathbb{E}}_{t}%
[|\hat{\xi}|^{\alpha}+\int_{t}^{T} |\hat{f}_{r}|^{\alpha}dr]+(\hat{\mathbb{E}%
}_{t}[\sup_{u\in[t,T]}|\hat{S}_{u}|^{\alpha}])^{\frac{1}{\alpha}}\Psi
_{t,T}^{\frac{\alpha-1}{\alpha}}\}.
\]
Letting $s=t$, we get the desired result.
\end{proof}

\subsection{Construction via penalization method}

For any $1\leq l\leq k$, consider the following $k$-dimensional $G$-BSDEs
parameterized by $n=1,2,\cdots$,
\begin{equation}
\label{penal 1}Y^{l,n}_{t}=\xi^{l}+\int_{t}^{T} f^{l}(s,Y_{s}^{n},Z_{s}%
^{l,n})ds+n\int_{t}^{T} (Y_{s}^{l,n}-S^{l}_{s})^{-}ds-\int_{t}^{T} Z_{s}%
^{l,n}dB_{s}-(K_{T}^{l,n}-K^{l,n}_{t}),
\end{equation}
where $Y^{n}=(Y^{1,n},\cdots, Y^{k,n})$. Let $L_{t}^{l,n}=\int_{0}^{t}
n(Y_{s}^{l,n}-S^{l}_{s})^{-} ds$. It is easy to check that $L^{l,n}$ is a
non-decreasing process and
\begin{equation}
\label{penal 2}Y^{l,n}_{t}=\xi^{l}+\int_{t}^{T} f^{l}(s,Y_{s}^{n},Z_{s}%
^{l,n})ds-\int_{t}^{T} Z_{s}^{l,n}dB_{s}-(K_{T}^{l,n}-K^{l,n}_{t}%
)+(L^{l,n}_{T}-L^{l,n}_{t}).
\end{equation}
The objective is to prove that for any $1\leq l\leq k$, the processes
$(Y^{l,n},Z^{l,n},A^{l,n})$ converge to $(Y^{l},Z^{l},A^{l})$, which is the
solution to the reflected $G$-BSDEs, where $A^{l,n}=L^{l,n}-K^{l,n}$. The
proof is similar with the one-dimensional case studied in \cite{LPSH}. The
main difference comes from the appearance of $Y^{j,n}$ in the generator $f$ in Equation \eqref{penal 1} with $j\neq l$, which leads to the modifications on the proofs for the estimate of $Y^{n},
Y^{n}-Y^{m}$ and also results in the lack of comparison theorem (which needs a kind of monotonicity condition as (i) of Theorem \ref{Myth2-3}).

\begin{lemma}
\label{est-YZKL} There exists a constant $C$ depending on $\alpha,T,L,k,G$,
but not on $n$, such that for $2\leq\alpha<\beta$ and $1\leq l\leq k$,
\[
\hat{\mathbb{E}}[\sup_{t\in[0,T]}|Y_{t}^{n}|^{\alpha}]\leq C,\text{ }
\hat{\mathbb{E}}[|K_{T}^{l,n}|^{\alpha}]\leq C,\text{ }\hat{\mathbb{E}}%
[|L_{T}^{l,n}|^{\alpha}]\leq C, \text{ }\hat{\mathbb{E}}[(\int_{0}^{T}
|Z_{t}^{l,n}|^{2} dt)^{\frac{\alpha}{2}}]\leq C.
\]

\end{lemma}

\begin{proof}
For any $r>0$, set $\tilde{Y}^{l,n}_{t}=|\bar{Y}^{l,n}_{t}|^{2}$, $\bar
{Y}^{l,n}_{t}=Y_{t}^{l,n}-I^{l}_{t}$ and $\bar{Z}^{l,n}_{t}=Z^{l,n}_{t}%
-\sigma^{l}(t)$. Note that for each $t\in[0,T]$, $(Y^{l,n}_{t}-I^{l}%
_{t})(Y^{l,n}_{t}-S^{l}_{t})^{-}\leq0$. Applying It\^{o}'s formula to
$\tilde{Y}_{t}^{\alpha/2}e^{rt}$ yields that
\[%
\begin{split}
&  \quad(\tilde{Y}_{t}^{l,n})^{\alpha/2}e^{rt}+\int_{t}^{T} re^{rs}(\tilde
{Y}_{s}^{l,n})^{\alpha/2}ds+\int_{t}^{T} \frac{\alpha}{2} e^{rs} (\tilde
{Y}_{s}^{l,n})^{\alpha/2-1}(\bar{Z}^{l,n}_{s})^{2}d\langle B\rangle_{s}\\
&  =|\xi^{l}|^{\alpha}e^{rT}+\alpha(1-\frac{\alpha}{2})\int_{t}^{T}%
e^{rs}(\tilde{Y}_{s}^{l,n})^{\alpha/2-2}(\bar{Y}^{l,n}_{s})^{2}(\bar{Z}%
^{l,n}_{s})^{2}d\langle B\rangle_{s}\\
&  \quad+\int_{t}^{T}{\alpha} e^{rs}(\tilde{Y}_{s}^{l,n})^{\alpha/2-1}\bar
{Y}^{l,n}_{s}(f^{(l,n)}_{s}+b^{l}(s))ds +\int_{t}^{T}\alpha e^{rs}(\tilde
{Y}_{s}^{l,n})^{\alpha/2-1}\bar{Y}^{l,n}_{s}dL_{s}^{l,n}\\
&  \quad- \int_{t}^{T}\alpha e^{rs}(\tilde{Y}_{s}^{l,n})^{\alpha/2-1}(\bar
{Y}^{l,n}_{s}\bar{Z}^{l,n}_{s}dB_{s}+\bar{Y}^{l,n}_{s}dK_{s}^{l,n}-\bar
{Y}^{l,n}_{s}dK^{I,l}_s)\\
&  \leq|\xi^{l}|^{\alpha}e^{rT}+ \alpha(1-\frac{\alpha}{2})\int_{t}^{T}%
e^{rs}(\tilde{Y}_{s}^{l,n})^{\alpha/2-2}(\bar{Y}^{l,n}_{s})^{2}(\bar{Z}%
^{l,n}_{s})^{2}d\langle B\rangle_{s}\\
&  \quad+\int_{t}^{T}{\alpha} e^{rs}(\tilde{Y}_{s}^{l,n})^{\alpha
/2-1/2}(|f^{(l,n)}_{s}|+b^{l}(s))ds-(M^{l,n}_{T}-M^{l,n}_{t}),
\end{split}
\]
where $f^{(l,n)}_{s}=f^{l}(s,\bar{Y}^{1,n}_{s}+I_{s}^{1},\cdots,\bar{Y}%
^{k,n}_{s}+I_{s}^{k},\bar{Z}^{l,n}_{s}+\sigma^{l}(s))$ and
\[
M^{l,n}_{t}=\int_{t}^{T}\alpha e^{rs}(\bar{Y}^{l,n}_{s})^{\alpha/2-1}(\bar
{Y}^{l,n}_{s}\bar{Z}^{l,n}_{s}dB_{s}+(\bar{Y}^{l,n}_{s})^{+}dK_{s}^{l,n}%
+(\bar{Y}^{l,n}_{s})^{-}dK^{I,l}_{s})
\]
is a $G$-martingale. By the assumption on $f^{l}$ and the Young inequality, we
have
\begin{equation}
\label{e2}%
\begin{split}
&  \int_{t}^{T}{\alpha} e^{rs}(\tilde{Y}^{l,n}_{s})^{\frac{\alpha-1}{2}%
}(|f_{s}^{(l,n)}|+b^{l}(s))ds\\
\leq &  \int_{t}^{T} {\alpha} e^{rs}(\tilde{Y}^{l,n}_{s})^{\frac{\alpha-1}{2}%
}(|f^{l}(s,\mathbf{0},0)|+L(|\bar{Z}^{l,n}_{s}|+|\sigma^{l}(s)|+\sum_{j=1}%
^{k}|\bar{Y}^{j,n}_{s}|+\sum_{j=1}^{k}|I^{j}_{s}|)+b^{l}(s))\\
\leq &  \int_{t}^{T} e^{rs}(|f^{l}(s,\mathbf{0},0)|^{\alpha}+|b^{l}%
(s)|^{\alpha}+L^{\alpha}(|\sigma^{l}(s)|^{\alpha}+\sum_{j=1,j\neq l}^{k}%
|\bar{Y}^{j,n}_{s}|^{\alpha}+\sum_{j=1}^{k}|I^{j}_{s}|^{\alpha})) ds\\
&  +\frac{\alpha(\alpha-1)}{4}\int_{t}^{T}e^{rs}(\tilde{Y}^{l,n}_{s}%
)^{\alpha/2-1}(\bar{Z}^{l,n}_{s})^{2}d\langle B\rangle_{s} +C(\alpha
,k,L,\underline{\sigma})\int_{t}^{T} e^{rs}(\tilde{Y}^{l,n}_{s})^{\alpha/2}ds.
\end{split}
\end{equation}
where $C(\alpha,k,L,\underline{\sigma})=2(k+1)(\alpha-1)+\alpha L+\frac{\alpha
L^{2}}{\underline{\sigma}^{2}(\alpha-1)}$. Set $r=C(\alpha
,k,L,\underline{\sigma})+1$. We obtain that
\begin{align*}
&  |\bar{Y}^{l,n}_{t}|^{\alpha}e^{rt}+M^{l,n}_{T}-M^{l,n}_{t}\\
\leq &  |\xi^{l}|^{\alpha}e^{rT}+\int_{t}^{T} e^{rs}(|f^{l}(s,\mathbf{0}%
,0)|^{\alpha}+|b^{l}(s)|^{\alpha}+L^{\alpha}(|\sigma^{l}(s)|^{\alpha}%
+\sum_{j=1,j\neq l}^{k}|\bar{Y}^{j,n}_{s}|^{\alpha}+\sum_{j=1}^{k}|I^{j}%
_{s}|^{\alpha})) ds\\
\leq &  C\{|\xi|^{\alpha}+\int_{t}^{T} (|f(s,\mathbf{0},0)|^{\alpha
}+|b(s)|^{\alpha}+|\sigma(s)|^{\alpha}+|\bar{Y}^{n}_{s}|^{\alpha}%
+|I_{s}|^{\alpha}) ds\}.
\end{align*}
Taking conditional expectations on both sides implies that
\[
|\bar{Y}^{l,n}_{t}|^{\alpha}\leq C\hat{\mathbb{E}}_{t}[|\xi|^{\alpha}+\int%
_{t}^{T} (|f(s,\mathbf{0},0)|^{\alpha}+|b(s)|^{\alpha}+|\sigma(s)|^{\alpha
}+|\bar{Y}^{n}_{s}|^{\alpha}+|I_{s}|^{\alpha}) ds].
\]
Summing up over $l$, we have
\begin{equation}
\label{e1}|\bar{Y}^{n}_{t}|^{\alpha}\leq C\hat{\mathbb{E}}_{t}[|\xi|^{\alpha
}+\int_{t}^{T} (|f(s,\mathbf{0},0)|^{\alpha}+|b(s)|^{\alpha}+|\sigma
(s)|^{\alpha}+|\bar{Y}^{n}_{s}|^{\alpha}+|I_{s}|^{\alpha}) ds].
\end{equation}
Taking expectations on both sides indicates that
\[
\hat{\mathbb{E}}[|\bar{Y}^{n}_{t}|^{\alpha}]\leq C\hat{\mathbb{E}}%
[|\xi|^{\alpha}+\int_{t}^{T} (|f(s,\mathbf{0},0)|^{\alpha}+|b(s)|^{\alpha
}+|\sigma(s)|^{\alpha}+|I_{s}|^{\alpha}) ds]+C\int_{t}^{T} \hat{\mathbb{E}%
}[|\bar{Y}^{n}_{s}|^{\alpha}]ds.
\]
Applying the Gronwall inequality, we get
\[
\hat{\mathbb{E}}[|\bar{Y}^{n}_{t}|^{\alpha}]\leq C\hat{\mathbb{E}}%
[|\xi|^{\alpha}+\int_{0}^{T} (|f(s,\mathbf{0},0)|^{\alpha}+|b(s)|^{\alpha
}+|\sigma(s)|^{\alpha}+|I_{s}|^{\alpha}) ds].
\]
Recalling Equation \eqref{e1} and Theorem \ref{the1.2}, there exists a
constant $C$ independent of $n$ such that for any $2\leq\alpha<\beta$,
$\hat{\mathbb{E}}[\sup_{t\in[0,T]}|\bar{Y}^{n}_{t}|^{\alpha}]\leq C$.
Consequently, $\hat{\mathbb{E}}[\sup_{t\in[0,T]}|{Y}^{n}_{t}|^{\alpha}]\leq
C$. By Proposition \ref{the1.6}, we have
\begin{align*}
\hat{\mathbb{E}}[(\int_{0}^{T} |Z^{l,n}_{s}|^{2}ds)^{\frac{\alpha}{2}}]  &
\leq C\{\hat{\mathbb{E}}[\sup_{s\in[0,T]}|Y^{l,n}_{s}|^{\alpha}]
+(\hat{\mathbb{E}}[\sup_{s\in[0,T]}|Y^{l,n}_{s}|^{\alpha}])^{1/2}%
(\hat{\mathbb{E}}[(\int_{0}^{T} |f^{(l,n)}(s,0,0)| ds)^{\alpha}])^{1/2}\},\\
\hat{\mathbb{E}}[|L^{l,n}_{T}-K^{l,n}_{T}|^{\alpha}]  &  \leq C\{\hat
{\mathbb{E}}[\sup_{s\in[0,T]}|Y^{l,n}_{s}|^{\alpha}]+\hat{\mathbb{E}}%
[(\int_{0}^{T} |f^{(l,n)}(s,0,0)| ds)^{\alpha}]\},
\end{align*}
where $f^{(l,n)}(s,0,0)=f^{l}(s,Y^{1,n}_{s},\cdots,Y^{l-1,n}_{s}%
,0,Y^{l+1,n}_{s},\cdots,Y^{k,n}_{s},0)$. Therefore, we get the desired uniform
estimates for $K^{l,n}$, $L^{l,n}$ and $Z^{l,n}$, respectively.
\end{proof}

Based on the uniform estimates obtained in Lemma \ref{est-YZKL}, by a similar analysis as the proof of Lemma 4.4 in \cite{LS}, we can obtain the following convergence property for $(Y^{l,n}-S^l)^-$, $1\leq l\leq k$. The only difference is in (12) of \cite{LS}. In our cases, $U=+\infty$, $f^{\varepsilon,n}(s)=f^l(s,Y_s^n,0)+m^{\varepsilon,n}_s$ and $Y^n$ here is $k$-dimensional while it is $1$-dimensional in \cite{LS}.
\begin{lemma}
\label{conv-Y-S} For any $2\leq\alpha<\beta$ and $1\leq l\leq k$,  we have
\[
\lim_{n\rightarrow\infty}\hat{\mathbb{E}}[\sup_{t\in[0,T]}|({Y}^{l,n}%
_{t}-S^{l}_{t})^{-}|^{\alpha}]=0.
\]

\end{lemma}

\begin{remark}
\upshape{Actually, Lemma 4.3 in \cite{LPSH} also indicates the uniform convergence property similar as in Lemma \ref{conv-Y-S} for the $1$-dimensional case. However, the proof needs the comparison theorem for $G$-BSDEs. In the multi-dimensional case, since we do not assume that the generator $f$ satisfy the monotonicity property (i) as in Theorem \ref{Myth2-3}, $Y^{l,n}$ may not increasing in $n$. Therefore, the method used in proving Lemma 4.3 in \cite{LPSH} cannot be applied to the multi-dimensional case. }
\end{remark}

We then show that $\{Y^n\}_{n\in\mathbb{N}}$ is a Cauchy sequence in $S_G^\alpha(0,T)$.
\begin{lemma}
\label{conv-Y} For any $2\leq\alpha<\beta$, we have
\[
\lim_{n,m\rightarrow\infty}\hat{\mathbb{E}}[\sup_{t\in[0,T]}|Y_{t}^{n}%
-Y_{t}^{m}|^{\alpha}]=0.
\]

\end{lemma}

\begin{proof}
For any $r>0$, set $\hat{Y}^{l}_{t}=Y_{t}^{l,n}-Y_{t}^{l,m}, \hat{Z}^{l}%
_{t}=Z_{t}^{l,n}-Z_{t}^{l,m}, \hat{K}^{l}_{t}=K_{t}^{l,n}-K_{t}^{l,m}, \hat
{L}^{l}_{t}=L_{t}^{l,n}-L_{t}^{l,m}$, $\bar{Y}^{l}_{t}=|\hat{Y}^{l}_{t}|^{2}$
and $\hat{f}^{l}_{t}=f^{l}(t,Y_{t}^{n},Z_{t}^{l,n})-f^{l}(t,Y_{t}^{m}%
,Z_{t}^{l,m})$. By applying It\^{o}'s formula to $|\bar{Y}_{t}^{l}|^{\alpha
/2}e^{rt}$, we get
\[%
\begin{split}
&  \quad|\bar{Y}_{t}^{l}|^{\alpha/2}e^{rt}+\int_{t}^{T} re^{rs}|\bar{Y}%
_{s}^{l}|^{\alpha/2}ds+\int_{t}^{T} \frac{\alpha}{2} e^{rs} |\bar{Y}_{s}%
^{l}|^{\alpha/2-1}(\hat{Z}^{l}_{s})^{2}d\langle B\rangle_{s}\\
&  = \alpha(1-\frac{\alpha}{2})\int_{t}^{T}e^{rs}|\bar{Y}_{s}^{l}%
|^{\alpha/2-2}(\hat{Y}^{l}_{s})^{2}(\hat{Z}^{l}_{s})^{2}d\langle B\rangle_{s}
+\int_{t}^{T}\alpha e^{rs}|\bar{Y}_{s}^{l}|^{\alpha/2-1}\hat{Y}^{l}_{s}%
d\hat{L}^{l}_{s}\\
&  \quad+\int_{t}^{T}{\alpha} e^{rs}|\bar{Y}_{s}^{l}|^{\alpha/2-1}\hat{Y}%
_{s}^{l}\hat{f}^{l}_{s}ds-\int_{t}^{T}\alpha e^{rs}|\bar{Y}_{s}^{l}%
|^{\alpha/2-1}(\hat{Y}^{l}_{s}\hat{Z}^{l}_{s}dB_{s}+\hat{Y}^{l}_{s}d\hat
{K}^{l}_{s})\\
&  \leq\alpha(1-\frac{\alpha}{2})\int_{t}^{T}e^{rs}|\bar{Y}_{s}^{l}%
|^{\alpha/2-2}(\hat{Y}^{l}_{s})^{2}(\hat{Z}^{l}_{s})^{2}d\langle B\rangle
_{s}+\int_{t}^{T}{\alpha} e^{rs}|\bar{Y}_{s}^{l}|^{\frac{\alpha-1}{2}}|\hat
{f}^{l}_{s}|ds-(M_{T}-M_{t})\\
&  \quad-\int_{t}^{T}\alpha e^{rs}|\bar{Y}_{s}^{l}|^{\alpha/2-1}(Y_{s}%
^{l,n}-S^{l}_{s})dL^{l,m}_{s}-\int_{t}^{T}\alpha e^{rs}|\bar{Y}_{s}%
^{l}|^{\alpha/2-1}(Y_{s}^{l,m}-S^{l}_{s})dL^{l,n}_{s},
\end{split}
\]
where $M_{t}=\int_{0}^{t} \alpha e^{rs}|\bar{Y}_{s}^{l}|^{\alpha/2-1}(\hat
{Y}^{l}_{s}\hat{Z}^{l}_{s}dB_{s}+(\hat{Y}^{l}_{s})^{+}dK_{s}^{l,m}+(\hat
{Y}^{l}_{s})^{-}dK_{s}^{l,n})$ is a $G$-martingale. Similar to \eqref{e2}, we
have
\[%
\begin{split}
\int_{t}^{T}{\alpha} e^{rs}|\bar{Y}_{s}^{l}|^{\frac{\alpha-1}{2}}|\hat{f}%
^{l}_{s}|ds \leq &  \int_{t}^{T} e^{rs}\sum_{j=1,j\neq l}^{k}|\hat{Y}^{j}%
_{s}|^{\alpha}ds +\frac{\alpha(\alpha-1)}{4}\int_{t}^{T}e^{rs}|\bar{Y}_{s}%
^{l}|^{\alpha/2-1}(\hat{Z}^{l}_{s})^{2}d\langle B\rangle_{s}\\
&  +((k-1)(\alpha-1)+\alpha L+\frac{\alpha L^{2}}{\underline{\sigma}%
^{2}(\alpha-1)})\int_{t}^{T} e^{rs}|\bar{Y}_{s}^{l}|^{\alpha/2}ds.
\end{split}
\]
Let $r=1+(k-1)(\alpha-1)+\alpha L+\frac{\alpha L^{2}}{\underline{\sigma}%
^{2}(\alpha-1)}$. By the above analysis, we have
\begin{align*}
|\bar{Y}_{t}^{l}|^{\alpha/2}e^{rt}+(M_{T}-M_{t})\leq &  -\int_{t}^{T}\alpha
e^{rs}|\bar{Y}_{s}^{l}|^{\alpha/2-1}(Y_{s}^{l,n}-S^{l}_{s})dL^{l,m}_{s}%
+\int_{t}^{T} e^{rs}\sum_{j=1,j\neq l}^{k}|\hat{Y}^{j}_{s}|^{\alpha}ds\\
&  -\int_{t}^{T}\alpha e^{rs}|\bar{Y}_{s}^{l}|^{\alpha/2-1}(Y_{s}^{l,m}%
-S^{l}_{s})dL^{l,n}_{s}.
\end{align*}
Taking conditional expectation on both sides of the above inequality, we
conclude that
\[%
\begin{split}
|\hat{Y}^{l}_{t}|^{\alpha}\leq &  C \hat{\mathbb{E}}_{t}[\int_{t}%
^{T}(m+n)|\bar{Y}_{s}^{l}|^{\alpha/2-1}(Y_{s}^{l,n}-S^{l}_{s})^{-}(Y_{s}%
^{l,m}-S^{l}_{s})^{-}ds+\int_{t}^{T} \sum_{j=1,j\neq l}^{k}|\hat{Y}^{j}%
_{s}|^{\alpha}ds]\\
\leq &  C\hat{\mathbb{E}}_{t}[\sum_{l=1}^{k}\int_{t}^{T}(m+n)|\bar{Y}%
_{s}|^{\alpha/2-1}(Y_{s}^{l,n}-S^{l}_{s})^{-}(Y_{s}^{l,m}-S^{l}_{s}%
)^{-}ds+\int_{t}^{T} |\hat{Y}_{s}|^{\alpha}ds].
\end{split}
\]
Summing up over $l$ yields
\begin{equation}
\label{e3}|\hat{Y}_{t}|^{\alpha}\leq C\hat{\mathbb{E}}_{t}[\sum_{l=1}^{k}%
\int_{t}^{T}(m+n)|\bar{Y}_{s}|^{\alpha/2-1}(Y_{s}^{l,n}-S^{l}_{s})^{-}%
(Y_{s}^{l,m}-S^{l}_{s})^{-}ds+\int_{t}^{T} |\hat{Y}_{s}|^{\alpha}ds].
\end{equation}
Taking expectations on both sides, we have
\[
\hat{\mathbb{E}}[|\hat{Y}_{t}|^{\alpha}]\leq C\hat{\mathbb{E}}[\sum_{l=1}%
^{k}\int_{0}^{T}(m+n)|\bar{Y}_{s}|^{\alpha/2-1}(Y_{s}^{l,n}-S^{l}_{s}%
)^{-}(Y_{s}^{l,m}-S^{l}_{s})^{-}ds]+C\int_{t}^{T} \hat{\mathbb{E}}[|\hat{Y}%
_{s}|^{\alpha}] ds.
\]
Applying the Gronwall inequality, it follows that
\[
\hat{\mathbb{E}}[|\hat{Y}_{t}|^{\alpha}]\leq C\hat{\mathbb{E}}[\sum_{l=1}%
^{k}\int_{0}^{T}(m+n)|\bar{Y}_{s}|^{\alpha/2-1}(Y_{s}^{l,n}-S^{l}_{s}%
)^{-}(Y_{s}^{l,m}-S^{l}_{s})^{-}ds].
\]
For any $1\leq\gamma<\beta/\alpha$, by the H\"{o}lder inequality, we have for
any $1\leq l\leq k$,
\begin{equation}
\label{e4}%
\begin{split}
&  \hat{\mathbb{E}}[(\int_{0}^{T}(m+n)|\bar{Y}_{s}|^{\alpha/2-1}(Y_{s}%
^{l,n}-S^{l}_{s})^{-}(Y_{s}^{l,m}-S^{l}_{s})^{-}ds)^{\gamma}]\\
\leq &  C\hat{\mathbb{E}}[\sup_{s\in[0,T]}|\hat{Y}_{s}|^{(\alpha-2)\gamma
}\{\sup_{s\in[0,T]}|(Y^{l,n}_{s}-S^{l}_{s})^{-}|^{\gamma}|L^{l,m}_{T}%
|^{\gamma}+\sup_{s\in[0,T]}|(Y^{l,m}_{s}-S^{l}_{s})^{-}|^{\gamma}|L^{l,n}%
_{T}|^{\gamma}\}]\\
\leq &  C(\hat{\mathbb{E}}[\sup_{s\in[0,T]}|\hat{Y}_{s}|^{\alpha\gamma
}])^{\frac{\alpha}{\alpha-2}}\{(\hat{\mathbb{E}}[\sup_{s\in[0,T]}|(Y^{l,n}%
_{s}-S^{l}_{s})^{-}|^{\alpha\gamma}])^{\frac{1}{\alpha}}(\hat{\mathbb{E}%
}[|L_{T}^{l,m}|^{\alpha\gamma}])^{\frac{1}{\alpha}}\\
&  +(\hat{\mathbb{E}}[\sup_{s\in[0,T]}|(Y^{l,m}_{s}-S^{l}_{s})^{-}%
|^{\alpha\gamma}])^{\frac{1}{\alpha}}(\hat{\mathbb{E}}[|L_{T}^{l,n}%
|^{\alpha\gamma}])^{\frac{1}{\alpha}}\},
\end{split}
\end{equation}
which converges to $0$ when $m,n$ go to infinity by Lemma \ref{est-YZKL} and
Lemma \ref{conv-Y-S}. Therefore, for any $2\leq\alpha<\beta$, we have
\begin{equation}
\label{e5}\lim_{n,m\rightarrow\infty}\hat{\mathbb{E}}[|Y^{n}_{t}-Y^{m}%
_{t}|^{\alpha}]=0.
\end{equation}
Combining \eqref{e3}-\eqref{e5} and applying Theorem \ref{the1.2}, we finally
get the desired result.
\end{proof}

Now, we are in a position to prove the main result, i.e., the existence and uniqueness of solutions to multi-dimensional reflected $G$-BSDEs.
\begin{proof}[Proof of Theorem \ref{MainThm}]
The uniqueness for $Y$ is a direct consequence of Proposition \ref{difference of Y}. Note that for each fixed $1\leq l\leq k$, $(Y^l,Z^l,A^l)$ can be seen as the solution to the 1-dimensional reflected $G$-BSDE with parameters $(\xi^l,\tilde{f}^l,S^l)$, where 
\begin{align*}
	\tilde{f}^l(t,y,z)=f^l(t,Y^1_t,\cdots,Y^{l-1}_t,y,Y^{l+1}_t,\cdots,Y^k_t,z).
\end{align*}
By the uniqueness for 1-dimensional reflected $G$-BSDEs, we obtain the uniqueness for the multi-dimensional case.

For the existence, by Lemma \ref{conv-Y}, there exists some $Y\in S_G^\alpha(0,T)$ such that 
\begin{align*}
	\lim_{n\rightarrow\infty}\hat{\mathbb{E}}[\sup_{t\in[0,T]}|Y_t-Y^n_t|^\alpha]=0.
\end{align*}
Besides, Lemma \ref{conv-Y-S} implies that for any $1\leq l\leq k$,  $Y_t^l\geq S_t^l$, $t\in[0,T]$. It remains to prove $(Z^{l,n},A^{l,n})$ converges to $(Z^{l},A^{l})$ and
	$\{-\int_{0}^{t} (Y^{l}_{s}-S^{l}_{s})dA^{l}_{s}\}_{t\in[0,T]}$ is a
	non-increasing $G$-martingale. The proof is similar with the one of Theorem
	5.1 in \cite{LPSH}, so we omit it.
\end{proof}

As a byproduct of  the penalization construction, we have the following comparison theorem for multi-dimensional
reflected $G$-BSDEs.

\begin{theorem}
\label{Myth4.1} Let $(\xi,f,S)$ and $(\bar
{\xi},\bar{f},\bar{S})$ be
two sets of data. Suppose that the coefficients $f$, $\bar{f}$, $\xi$, $\bar{\xi}$, $S$, $\bar{S}$
satisfy $(A{1})-(A{4}).$ Assume the following conditions hold:

\begin{itemize}
\item[(i)] for each $1\leq l\leq k$, $f^{l}(t,y,z^{l})\geq\bar{f}^{l}%
(t,\bar{y},z^{l})$ if
$z^{l}\in\mathbb{R}$ and $y,\bar{y}\in\mathbb{R}^{k}$ satisfying $y^{j}%
\geq\bar{y}^{j}$ for $j\neq l$ and $y^{l}=\bar{y}^{l}$;

\item[(ii)] $\xi\geq\bar{\xi}$;

\item[(iii)] $S_{t}\geq\bar{S}_{t}$, $0\leq t\leq T$, $q.s.$.
\end{itemize}
Suppose that $(Y,Z,A)$ and $(\bar{Y},\bar{Z},\bar{A})$ are the solutions to the
reflected $G$-BSDE with the above two  sets of parameters, respectively. Then, we have%
\[
Y_{t}\geq\bar{Y}_{t},\quad0\leq t\leq T,\quad q.s.
\]

\end{theorem}

\begin{proof}
Consider the following $G$-BSDEs parameterized by $n=1,2,\cdots$,%
\[%
\begin{split}
\bar{Y}_{t}^{l,n}  &  =\bar{\xi}^{l}+\int_{t}^{T}\bar{f}^{l}(s,Y_{s}^{n}%
,Z_{s}^{l,n})ds+n\int_{t}^{T}(\bar{Y}_{s}^{l,n}-\bar{S}_{s}^{l})^{-}ds-\int%
_{t}^{T}\bar{Z}_{s}^{l,n}dB_{s}\\
&  -(\bar{K}_{T}^{l,n}-\bar{K}_{t}^{l,n}),\ \ \ 1\leq l\leq k.
\end{split}
\]
Similar analysis as the proof of Theorem \ref{MainThm}, we can show that
$\lim_{n\rightarrow\infty}\hat{\mathbb{E}}[\sup_{t\in\lbrack0,T]}|\bar{Y}%
_{t}-\bar{Y}_{t}^{n}|^{\alpha}]=0$, where $2\leq\alpha<\beta$. Note
that $(Y,Z,A)$ is the solution of the reflected $G$-BSDE with parameters
$(\xi,f,S)$ and $Y_{t}\geq S_{t}$, $0\leq
t\leq T$. Thus we have
\[
Y_{t}^{l}=\xi^{l}+\int_{t}^{T}f^{l}(s,Y_{s},Z_{s}^{l})ds+\int_{t}^{T}n(Y_{s}^{l}%
-S_{s}^{l})^{-}ds-\int_{t}^{T}Z_{s}^{l}dB_{s}+(A_{T}^{l}-A_{t}^{l}).
\]
According to Theorem \ref{Myth2-3}, we have  $Y^{l}_{t}\geq\bar{Y}%
_{t}^{l,n}$, for each $n\in\mathbb{N}$. Now letting $n\rightarrow\infty$, we
conclude that $Y^{l}_{t}\geq\bar{Y}_{t}^{l}$, as desired.
\end{proof}
\begin{remark}
	\upshape{
To guarantee the comparison theorem to hold for multi-dimensional BSDEs with diagonal generators, the condition (i) in Theorem \ref{Myth2-3} (i.e., the condition (i) in Theorem \ref{Myth4.1}) is usually imposed.  Indeed,
in the linear conditional expectation case, it is necessary and
sufficient for the comparison theorem of multi-dimensional BSDEs with
diagonal generators to hold.

To be more detailed, on the classical probability space,  suppose that the generators $f^{l}(t,y,z^{l})$ and $\bar{f}^{l}(t,\bar
{y},z^{l})$, for
$1\leq l\leq k$, are diagonal and satisfy standard assumptions of BSDEs  (see \cite{PP}). Consider the following two $k$-dimensional BSDEs on $[0,T]$:
\[
Y_{t}^{l}=\xi^{l}+\int_{t}^{T}f^{l}(s,Y_{s},Z_{s}^{l})ds-\int_{t}^{T}Z_{s}%
^{l}dB_{s},
\]
and
\[
\bar{Y}_{t}^{l}=\bar{\xi}^{l}+\int_{t}^{T}\bar{f}^{l}(s,\bar{Y}_{s},\bar
{Z}_{s}^{l})ds-\int_{t}^{T}\bar{Z}_{s}^{l}dB_{s},
\]
\ where $1\leq l\leq k$. Then the following are equivalent:

\begin{itemize}
	\item[(i)] for any stopping time $\tau\leq T,$ $\xi,\bar{\xi}\in L^{2}(\mathcal{F}_{\tau
	},\mathbb{R}^{k})$ such that $\xi\geq\bar{\xi},$ the solotion $(Y,Z)$ and
	$(\bar{Y},\bar{Z})$ to the above two BSDEs, respectively, on $[0,\tau]$
	satsify $Y_{t}\geq\bar{Y}_{t}$, $t\in\lbrack0,\tau]$;
	
	\item[(ii)] for any $z^{l}\in\mathbb{R}$ and $y,\bar{y}\in\mathbb{R}^{k}$
	satisfying $y^{j}\geq\bar{y}^{j}$ for $j\neq l$ and $y^{l}=\bar{y}^{l}$, it
	holds that $f^{l}(t,y,z^{l})\geq\bar{f}^{l}(t,\bar{y},z^{l})$, for $1\leq
	l\leq k$.
\end{itemize}

\noindent It can be proved by a similar argument as  in Theorem 2.2 in \cite{HTW2022} as follows:

\textit{Step 1.} Assume (i) hold. From Theorem 2.1 in \cite{HTW2022}, we know that (i)
is equivalent to: for each  $y,\bar{y}\in\mathbb{R}^{k}$ and
$z,\bar{z}\in\mathbb{R}^{k},$
\begin{equation}
	-4\langle y^{-},f(t,y^{+}+\bar{y},z)-\bar{f}(t,\bar{y},\bar{z})\rangle
	_{\mathbb{R}^{k}}\leq2\sum_{l=1}^{k}I_{\{y^{l}<0\}}|z^{l}-\bar{z}^{l}%
	|^2+C|y^{-}|^2.\label{Eq111}%
\end{equation}
Then for any fixed $l$, we take $\delta_{l}{y}\in\mathbb{R}^{k}$ satisfying
$\delta_{l}{y}\geq0$ and $(\delta_{l}{y})^{l}=0$. We then take, for any
$\varepsilon>0$,
\[
y=\delta_{l}{y}-\varepsilon e^{l}.
\]
Plug this into (\ref{Eq111}) and let $z=\bar{z}$, we get
\[
-4\varepsilon[f^{l}(t,\delta_{l}{y}+\bar{y},\bar{z}^{l})-\bar{f}^{l}(t,\bar
{y},\bar{z}^{l})]\leq C\varepsilon^{2}.
\]
Divide on both sides by $-\varepsilon$, we get
\[
4[f^{l}(t,\delta^{l}{y}+\bar{y},\bar{z}^{l})-\bar{f}^{l}(t,\bar{y},\bar{z}%
^{l}]\geq-C\varepsilon.
\]
Letting $\varepsilon\rightarrow0$, we get (ii).

\textit{Step 2.} Now assume (ii) hold. For each $l,$ we take
$
\delta_{l}{y}^{+}\in\mathbb{R}^{k}%
$
satisfying $(\delta_{l}{y}^{+})^{j}=(y^{+})^{j},j\neq l$ and $(\delta_{l}%
{y}^{+})^{l}=0.$ From (ii),
\begin{align*}
	&  -4\langle y^{-},f(t,y^{+}+\bar{y},z)-\bar{f}(t,\bar{y},\bar{z}%
	)\rangle_{\mathbb{R}^{k}}\\
	&  =\sum_{l=1}^{k}-4(y^{-})^{l}[f^{l}(t,y^{+}+\bar{y},z^{l})-\bar{f}%
	^{l}(t,\bar{y},\bar{z}^{l})]\\
	&  =\sum_{l=1}^{k}-4(y^{-})^{l}[(f^{l}(t,y^{+}+\bar{y},z^{l})-f^{l}%
	(t,\delta_{l}{y}^{+}+\bar{y},\bar{z}^{l}))+(f^{l}(t,\delta_{l}y^{+}+\bar
	{y},\bar{z}^{l})-\bar{f}^{l}(t,\bar{y},\bar{z}^{l}))]\\
	&  \leq\sum_{l=1}^{k}-4(y^{-})^{l}[f^{l}(t,y^{+}+\bar{y},z^{l})-f^{l}%
	(t,\delta_{l}{y}^{+}+\bar{y},\bar{z}^{l})]
\end{align*}
Then applying the Lipschitz assumption, we further obtain
\begin{align*}
	&  -4\langle y^{-},f(t,y^{+}+\bar{y},z)-\bar{f}(t,\bar{y},\bar{z}%
	)\rangle_{\mathbb{R}^{k}}\leq\sum_{l=1}^{k}4C(y^{-})^{l}[|y^{+}-\delta_{l}%
	{y}^{+}|+|z^{l}-\bar{z}^{l}|]\\
	&  \ \ \ \ \ \ \ \ \ \ \ \ \ \ \ \ \ \ \ \ =\sum_{l=1}^{k}4C(y^{-})^{l}%
	[(y^{+})^{l}+|z^{l}-\bar{z}^{l}|]\\
	&  \ \ \ \ \ \ \ \ \ \ \ \ \ \ \ \ \ \ \ \ =\sum_{l=1}^{k}4C(y^{-})^{l}%
	|z^{l}-\bar{z}^{l}|\\
	&  \ \ \ \ \ \ \ \ \ \ \ \ \ \ \ \ \ \ \ \ \leq\sum_{l=1}^{k}[2|z^{l}-\bar
	{z}^{l}|^{2}I_{\{y^{l}<0\}}+C((y^{-})^{l})^{2}]\\
	&  \ \ \ \ \ \ \ \ \ \ \ \ \ \ \ \ \ \ \ \ =2\sum_{l=1}^{k}|z^{l}-\bar{z}%
	^{l}|^{2}I_{\{y^{l}<0\}}+C|y^{-}|^{2}.
\end{align*}
The proof is complete.

Moreover, a counterexample when such a condition is violated can be found in
Example 3.2 in \cite{WX}.}
\end{remark}

\subsection{Picard iteration method}
In this subsection, we will give another proof of Theorem \ref{MainThm}, based
on a fixed point argument. 
We denote by $M_{G}^{\alpha}(a,b;\mathbb{R}^{k})$, $S_{G}^{\alpha}(a,b;\mathbb{R}^{k})$ and $\mathcal{S}_{G}^{\alpha
}(a,b;\mathbb{R}^{k})$  the corresponding spaces for the
stochastic processes defined on the time interval $[a,b]$.

We have the following well-posednesss result on the local solution to
reflected $G$-BSDE \eqref{my1}.

\begin{theorem}
\label{my16} Assume that (A1)-(A4) hold for some $\beta>2$. Then there exists
a constant $0<\delta\leq T$ depending only on $T,G,k,\beta$ and $L$ such that
for any $h\in(0,\delta]$, $t\in\lbrack0,T-h]$ and given $\zeta\in L_{G}%
^{\beta}(\Omega_{t+h};\mathbb{R}^{k})$ such that $\zeta\geq S_{t+h}$, the $G$-BSDE with parameters $(\zeta,f,S)$ 
 on the interval $[t,t+h]$
admits a unique solution $(Y,Z,A)\in\mathcal{S}_{G}^{\alpha}(t,t+h;\mathbb{R}%
^{k})$ for each $2\leq\alpha<\beta$. Moreover, $Y\in{M}_{G}^{\beta
}(t,t+h;\mathbb{R}^{k}).$
\end{theorem}

To prove Theorem \ref{my16}, let us consider the following $G$-BSDE with
reflection $S$ on the interval $[t,t+h]$, for any $h\in(0,\delta]$, $t\in\lbrack0,T-h]$:
\begin{equation}\begin{cases}\label{myq4}%
Y_{s}^{U,l}=\zeta^{l}+\int_{s}^{t+h}f^{l,U}(r,Y_{r}^{U,l},Z_{r}^{U,l}%
)dr-\int_{s}^{t+h}Z_{r}^{U,l}dB_{r}+(A_{t+h}^{U,l}-A_{s}^{U,l}), 1\leq
l\leq k, \\
Y^{U,l}_s\geq S^l_s, \ s\in[t,t+h],  1\leq l\leq k,\\ \{-\int_t^s(Y^{U,l}_r-S^l_r)dA^{U,l}_r\}_{s\in[t,t+h]} \textrm{ is a non-increasing $G$-martingale}, 1\leq l\leq k,
\end{cases}\end{equation}
where $U\in{M}_{G}^{\beta}(t,t+h;\mathbb{R}^{k})$, $\zeta\in L_{G}^{\beta
}(\Omega_{t+h};\mathbb{R}^{k})$ and
\[
f^{l,U}(t,y^{l},z^{l})=f^{l}(t,U_{t}^{1},\cdots,U_{t}^{l-1},y^{l},U_{t}%
^{l+1},\cdots,U_{t}^{k},z^{l}).
\]
We denote $X^{U}=(X^{U,1},\cdots,X^{U,k})$ for $X=Y,Z,A$.  
\begin{lemma}
\label{myq6}  Given $U\in{M}_{G}^{\beta}(t,t+h;\mathbb{R}^{k})$ and $\zeta\in L_{G}^{\beta
}(\Omega_{t+h};\mathbb{R}^{k})$, for any $2\leq \alpha<\beta$, the reflected $G$-BSDE \eqref{myq4} has
a unique solution $(Y^{U},Z^{U},K^{U})$ in $\mathcal{S}_{G}^{\alpha
}(t,t+h;\mathbb{R}^{k}),$ and moreover, $Y^{U}\in{M}_{G}^{\beta}%
(t,t+h;\mathbb{R}^{k})$.
\end{lemma}
\begin{proof}
Fix any $l\geq1$. From Lemma 3.3 in \cite{Liu-stochastics}, we know that
	$f^{l,U}(s,y^{l},z^{l})\in M_{G}^{\beta}(t,t+h)$ for each $y^{l}\in
	\mathbb{R},z^{l}\in\mathbb{R}^{d}$.  Then applying Theorem \ref{the1.14}  to the $l$-th
	components of  reflected $G$-BSDE (\ref{myq4}), we  get the unique solution $(Y_{t}%
	^{U,l},Z_{t}^{U,l},A_{t}^{U,l}),$ which constitutes the  unique solution to \eqref{myq4} when $l$ run through $1$ to $k$.
	
	Next we show that $Y^{U,l}\in{M}_{G}^{\beta}(t,t+h)$. By Proposition \ref{the1.14}, we have for all $s\in\lbrack
	t,t+h]$ that
	\[
	|Y_{s}^{U,l}|^{\beta}\leq C\mathbb{\hat{E}}_{s}[|{\zeta^l}|^{\beta}+\int_{t}%
	^{t+h}(|f^{U,l}(r,0,0)|^{\beta}+|b^{I,l}(r)|^{\beta}+|\sigma^{I,l}(r)|^{\beta
	})dr+\sup_{r\in\lbrack t,t+h]}|I^l_{r}|^{\beta}]=:\rho^l_{s},
	\]
	Then following the same steps as those in Lemma 3.2 in \cite{Liu-stochastics}, we
	can derive the desired result.
\end{proof}

Thus, we can define a map $\Gamma:U\rightarrow\Gamma(U)$ from ${M}%
_{G}^{\beta}(t,t+h;\mathbb{R}^{k})$ to ${M}_{G}^{\beta}(t,t+h;\mathbb{R}^{k})$
by
\[
\Gamma(U):=Y^{U},\ \ \ \ \text{for each }U\in{M}_{G}^{\beta}(t,t+h;\mathbb{R}%
^{k}).
\]
In the following, we show the map $\Gamma$ is a contraction if we take sufficiently small $h$.

\begin{lemma}
\label{myq9} There exists some constant $\delta>0$ depending only on
$T,G,k,\beta$ and $L$ such that for each $h\in(0,\delta]$,
\[
\Vert{Y}^{U}-{Y}^{\bar{U}}\Vert_{M_{G}^{\beta}(t,t+h;\mathbb{R}^{k})}\leq
\frac{1}{2}\Vert U-\bar{U}\Vert_{M_{G}^{\beta}(t,t+h;\mathbb{R}^{k}%
)},\ \ \ \ \text{for each}\ U,\bar{U}\in M_{G}^{\beta}(t,t+h;\mathbb{R}^{k}).
\]

\end{lemma}

\begin{proof}
For each fixed $1\leq l\leq k$, by applying Proposition
\ref{the1.10} to $(Y^{U,l}-{Y}^{\bar{U},l})$, we obtain that
\[
|Y_{s}^{U,l}-{Y}_{s}^{\bar{U},l}|^{\beta}\leq C\mathbb{\hat{E}}_{s}[\int%
_{s}^{t+h}|\hat{h}^l_{r}|^\beta dr],\ \ \forall s\in\lbrack t,t+h],
\]
where
\[
\hat{h}^l_{s}=|f^{l,U}(s,Y_{s}^{\bar{U},l},Z_{s}^{\bar{U},l})-f^{l,\bar{U}%
}(s,Y_{s}^{\bar{U},l},Z_{s}^{\bar{U},l})|.
\]
By Assumption (A2), we then have
\[
\mathbb{\hat{E}}[|Y_{s}^{U,l}-{Y}_{s}^{\bar{U},l}|^{\beta}]\leq C\mathbb{\hat{E}}[\int_{t}^{t+h}|U_{r}-\bar{U}_{r}|^{\beta}dr],\ \ \forall
s\in\lbrack t,t+h].
\]
Summing over $l,$ we deduce that for $s\in\lbrack t,t+h]$,
\[
\mathbb{\hat{E}}[|Y_{s}^{U}-{Y}_{s}^{\bar{U}}|^{\beta}]\leq C
\sum\limits_{l=1}^{k}\mathbb{\hat{E}}[|Y_{s}^{U,l}-{Y}_{s}^{\bar{U},l}%
|^{\beta}]\leq C\mathbb{\hat{E}}[\int_{t}^{t+h}|U_{r}%
-\bar{U}_{r}|^{\beta}dr].
\]
Thus,
\[
\Vert{Y}^{U}-{Y}^{\bar{U}}\Vert_{M_{G}^{\beta}(t,t+h;\mathbb{R}^{k})}\leq
|\int_{t}^{t+h}\mathbb{\hat{E}}[|Y_{s}^{U}-{Y}_{s}^{\bar{U}}|^{\beta
}]ds|^{\frac{1}{\beta}}\leq Ch^{\frac{1}{\beta}}\Vert U-\bar{U}\Vert_{M_{G}^{\beta
}(t,t+h;\mathbb{R}^{k})}.
\]
Then we can take $\delta>0$ small enough such that, for each $h\in(0,\delta
]$,
\[
\Vert{Y}^{U}-{Y}^{\bar{U}}\Vert_{M_{G}^{\beta}(t,t+h;\mathbb{R}^{k})}\leq
\frac{1}{2}\Vert U-\bar{U}\Vert_{M_{G}^{\beta}(t,t+h;\mathbb{R}^{k})}.
\]
The proof is complete.
\end{proof}

Now we can state the proof for the existence and uniqueness of local solutions
as follows.

\begin{proof}
[The proof of Theorem \ref{my16}]   To prove the result, we shall use a fixed point argument for the map  $\Gamma:{M}_{G}^{\beta}(t,t+h;\mathbb{R}%
^{k})\rightarrow {M}_{G}^{\beta}(t,t+h;\mathbb{R}%
^{k})$. We first prove the existence. From Lemma \ref{myq9}, we see that $\Gamma$ is a contraction,
and thus has a unique fixed point $Y\in {M}_{G}^{\beta}(t,t+h;\mathbb{R}%
^{k})$ such that $\Gamma(Y)=Y$. Moreover, for $U=Y$, from Lemma \ref{myq6} we know that there exists solution $(Y^Y,Z^Y,A^Y)\in \mathcal{S}_{G}^{\alpha
}(t,t+h;\mathbb{R}^{k})$ that solves \eqref{myq4}. Combining the above analysis, we deduce that $(Y,Z^Y,A^Y)$ solves \eqref{myq4}, which can be written as follows:
 for $ 1\leq l\leq k$,  it holds on $[t,t+h]$ that
\[
Y_{s}^{l}=\zeta^{l}+\int_{s}^{t+h}f^{l}(r,Y_{r},Z_{r}^{Y,l} )dr-\int_{s}%
^{t+h}Z_{r}^{Y,l}dB_{r}+(A_{t+h}^{Y,l}-A_{s}^{Y,l}),
\]
with
\[
Y^l_s\geq S^l_s\quad
\text{and} \quad -\int_{t}^{s} (Y_{r}^{l}-S_{r}^{l})dA_{r}^{l} \ \text{being non-increasing $G$-martingale}.
\]
That means that $(Y,Z^{Y},A^{Y})$ is the solution of reflected $G$-BSDE with parameters $(\zeta,f,S)$ on time interval $[t,t+h]$.

We then consider the uniqueness. If $(Y',Z',A')\in\mathcal{S}_{G}^{\alpha}%
(t,t+h;\mathbb{R}^{k})$ is also the solution of reflected $G$-BSDE  with parameters $(\zeta,f,S)$ on time interval $[t,t+h]$, then $Y'\in {M}_{G}^{\beta}(t,t+h;\mathbb{R}%
^{k})$ and $\Gamma(Y')=Y'$, which means that $Y'$ is the fixed point of $\Gamma$. So $Y=Y'.$
Then for any $1\leq l\leq k$, $(Y^l,Z^{Y,l},A^{Y,l})$ and $(Y'^{,l},Z'^{,l},A'^{,l})$ are both solutions of \eqref{myq4} with $U=Y$. Applying Lemma \ref{myq6} implies that $(Z^Y,A^Y)=(Z',A')$.
\end{proof}

By a backward iteration on local solutions, we finally have the following new
proof for the global well-posedness Theorem \ref{MainThm} for reflected
$G$-BSDE (\ref{my1}).

\begin{proof}
We choose the constant $\delta$ determined in Theorem \ref{my16} and take an integer
$m$ large enough such that $m\delta\geq T$. We can then take $h=\frac{T}{m}$
and apply Theorem \ref{my16} to deduce that the reflected $G$-BSDE  with parameters $(\xi,f,S)$ on $[T-h,T]$ 
admits a unique solution $(Y^{(m)},Z^{(m)},A^{(m)})\in\mathcal{S}_{G}^{\alpha
}(T-h,T;\mathbb{R}^{k})$. Choosing $T-{h}$ as the terminal time
and $Y_{T-h}^{(m)}$ as the terminal value, we apply Theorem \ref{my16}
again to obtain that the reflected $G$-BSDE (\ref{Myeq3-1}) with parameters $(Y^{(m)}_{T-h},f,S)$ on $[T-2{h},T-{h}]$ has a unique solution 
$(Y^{(m-1)},Z^{(m-1)},A^{(m-1)})\in\mathcal{S}_{G}^{\alpha}%
(T-2h,T-h;\mathbb{R}^{k})$. Finally we obtain a sequence
$(Y^{(i)},Z^{(i)},K^{(i)})_{i\leq m}$ by repeating this procedure. Let us
take
\[
{Y}_{t}=\sum\limits_{i=1}^{m}Y_{t}^{(i)}I_{[(i-1)h,ih)}(t)+Y_{T}%
^{(m)}I_{\{T\}}(t),\quad {Z}_{t}=\sum\limits_{i=1}^{m}Z_{t}^{(i)}I_{[(i-1)h,ih)}%
(t)+Z_{T}^{(m)}I_{\{T\}}(t)
\]
and
\[
A_{t}=A_{t}^{(i)}+\sum_{j={1}}^{i-1}A_{jh}^{(j)},\ \text{for}%
\ t\in\lbrack(i-1)h,ih), i=1,\cdots,m\quad\text{and}\quad A_{T}=A_{T}^{(m)}%
+\sum_{j=1}^{m-1}A_{jh}^{(j)},
\]
where we make the convention that $\sum_{j={1}}^{0}A_{jh}^{(j)}=0$. Then it is easy to see that
$({Y},{Z},{K})\in\mathcal{S}_{G}^{\alpha}(0,T;\mathbb{R}^{k})$ is a solution
to reflected $G$-BSDE with parameters $(\xi,f,S)$ and $Y\in{M}_{G}^{\beta}(0,T;\mathbb{R}^{k})$.

The uniqueness result on the whole interval follows
from the one on each small time interval, which completes the proof.
\end{proof}

\section{Multi-dimensional obstacle problems for fully nonlinear PDEs}

In this section, we establish the connection between the multi-dimensional
reflected $G$-BSDEs and the system of fully nonlinear PDEs with obstacle
constraints. Roughly speaking, the solution of a multi-dimensional reflected
$G$-BSDE in a Markovian setting coincides with the unique viscosity solution
to a system of fully nonlinear PDE with obstacle constraints. 

For this purpose, for any fixed $t\in[0,T]$ and $\eta\in L_{G}^{p}%
(\Omega_{t};\mathbb{R}^{n})$ with $p\geq 2$, let us first consider the
following $G$-SDEs:
\begin{equation}
\label{GSDE}%
\begin{cases}
d X^{t,\eta}_{s}=b(s,X^{t,\eta}_{s})ds+h(s,X^{t,\eta}_{s})d\langle
B\rangle_{s}+\sigma(s,X^{t,\eta}_{s})dB_{s}, \ s\in[t,T],\\
X^{t,\eta}_{t}=\eta,
\end{cases}
\end{equation}
where $b,h,\sigma:[0,T]\times\mathbb{R}^{n}\rightarrow\mathbb{R}^{n}$ are
deterministic continuous functions satisfying

\begin{itemize}
\item[(Ai)] there exists a positive constant $L$ such that, for any
$t\in[0,T]$ and any $x,x^{\prime}\in\mathbb{R}^{n}$
\[
|b(t,x)-b(t,x^{\prime})|+|h(t,x)-h(t,x^{\prime})|+|\sigma(t,x)-\sigma
(t,x^{\prime})|\leq L|x-x^{\prime}|.
\]

\end{itemize}

By a similar analysis as in Chapter 5 of \cite{P10}, the $G$-SDE
\eqref{GSDE} has a unique solution $X^{t,\eta}\in M^{p}_{G}(t,T;\mathbb{R}%
^{n})$. Furthermore, the following estimates hold.

\begin{proposition}
[\cite{P10}]\label{the1.17} Let $\eta,\eta^{\prime}\in L_{G}^{p}(\Omega
_{t};\mathbb{R}^{n})$ and $p\geq2$. Then we have, for each $\delta\in
[0,T-t]$,
\begin{align*}
\hat{\mathbb{E}}_{t}[\sup_{s\in[t,t+\delta]}|X_{s}^{t,\eta}-X_{s}%
^{t,\eta^{\prime}}|^{p}]  &  \leq C|\eta-\eta^{\prime}|^p,\\
\hat{\mathbb{E}}_{t}[|X_{t+\delta}^{t,\eta}|^{p}]  &  \leq C(1+|\eta|^{p}),\\
\hat{\mathbb{E}}_{t}[\sup_{s\in[t,t+\delta]}|X_{s}^{t,\eta}-\eta|^{p}]  &  \leq
C(1+|\eta|^{p})\delta^{p/2},
\end{align*}
where the constant $C$ depends on $L,G,p,n$ and $T$.
\end{proposition}

Now, let us consider the following $k$-dimensional reflected
$G$-BSDEs on the interval $[t,T]$: 
\begin{equation}
\label{decoupledRGBSDE}%
\begin{cases}
-dY^{t,\eta;i}_{s}=f^{i}(s,X_{s}^{t,\eta},Y_{s}^{t,\eta},Z_{s}^{t,\eta;i})ds
+g^{i}(s,X_{s}^{t,\eta},Y_{s}^{t,\eta},Z_{s}^{t,\eta;i})d\langle B\rangle
_{s}\\
\ \ \ \ \ \ \ \ \ \ -Z_{s}^{t,\eta;i}dB_{s}+dA^{t,\eta;i}_{s},\\
Y^{t,\eta;i}_{T}=\phi^{i}(X^{t,\eta}_{T}), \ 
Y^{t,\eta;i}_{s}\geq l^{i}(s,X_{s}^{t,\eta}),\ s\in[t,T],\\
\{-\int_{t}^{s}(Y^{t,\eta;i}_{r}- l^{i}(r,X_{r}^{t,\eta}))dA^{t,\eta;i}%
_{r}\}_{s\in[t,T]} \text{ is a non-increasing $G$-martingale},
\end{cases}
\end{equation}
where $1\leq i\leq k$ and the functions $\phi^{i}:\mathbb{R}^{n}%
\rightarrow\mathbb{R}$, $f^{i},g^{i}:[0,T]\times\mathbb{R}^{n}\times
\mathbb{R}^{k}\times\mathbb{R}\rightarrow\mathbb{R}$, $l^{i}:[0,T]\times
\mathbb{R}^{n}\rightarrow\mathbb{R}$ are continuous, deterministic functions
satisfying the following assumptions:

\begin{itemize}
\item[(Aii)] there exists a constant $L>0$ such that for any $1\leq i\leq k$,
$x_{1},x_{2}\in\mathbb{R}^{n}$, $y_{1},y_{2}\in\mathbb{R}^{k}$, $z_{1}%
,z_{2}\in\mathbb{R}$
\begin{align*}
&  |\phi^{i}(x_{1})-\phi^{i}(x_{2})|\leq L|x_{1}-x_{2}|,\ |l^{i}%
(t,x_{1})-l^{i}(t,x_{2})|\leq L|x_{1}-x_{2}|,\\
&  |f^{i}(t,x_{1},y_{1},z_{1})-f^{i}(t,x_{2},y_{2},z_{2})|\leq L(|x_{1}%
-x_{2}|+|y_{1}-y_{2}|+|z_{1}-z_{2}|),\\
&  |g^{i}(t,x_{1},y_{1},z_{1})-g^{i}(t,x_{2},y_{2},z_{2})|\leq L(|x_{1}%
-x_{2}|+|y_{1}-y_{2}|+|z_{1}-z_{2}|){\color{blue};}
\end{align*}

\item[(Aiii)] for any $1\leq i\leq k$, $x\in\mathbb{R}^n$ and $t\in[0,T]$,
$l^{i}(t,x)\leq\widetilde{l}^{i}(t,x)$ and $l^{i}(T,x)\leq\phi^{i}(x)$, where
$\widetilde{l}^{i}$ belongs to the space $C^{1,2}_{Lip}([0,T]\times
\mathbb{R}^{n})$, which is the collection of all functions of class
$C^{1,2}([0,T]\times\mathbb{R}^{n})$ whose partial derivatives of order less
than or equal to $2$ and itself are continuous in $t$ and Lipschitz continuous
in $x$;
\item[(Aiv)] for $1\leq i\leq k$, for any $z\in \mathbb{R}$ and $y,\bar{y}\in \mathbb{R}^{k}$ satisfying $y^{j}%
\geq\bar{y}^{j}$ for $j\neq i$ and $y^{i}=\bar{y}^{i}$, it holds $f^{i}(t,y,z)\geq{f}^{i}(t,\bar{y},z)$ and $g^{i}(t,y,z)\geq{g}^{i}(t,\bar{y},z)$.
\end{itemize}

By Theorem \ref{MainThm}, for any $\eta\in L_{G}^{\beta}(\Omega_{t}%
;\mathbb{R}^{n})$ with $\beta>2$, there exists a unique solution $(Y^{t,\eta},Z^{t,\eta
},A^{t,\eta})\in\mathcal{S}_{G}^{\alpha}(t,T;\mathbb{R}^{k})$ to reflected
$G$-BSDE \eqref{decoupledRGBSDE}, where $2\leq\alpha<\beta$. For each fixed
$(t,x)\in\lbrack0,T]\times\mathbb{R}^{n}$ and $1\leq i\leq k$, we define:
\begin{equation}
u^{i}(t,x):=Y_{t}^{t,x;i}. \label{valuefunction}%
\end{equation}
Note that $u^{i}$ is a deterministic function. By Proposition
\ref{estimate for Y}, \ref{difference of Y} and \ref{the1.17}, it is easy to to check that
there exists a constant $C$ depending on $L,G,n,k$ and $T$, such that
\begin{align*}
	|u(t,x)|  &  \leq C(1+|x|),\\
	|u(t,x)-u(t,x^{\prime}) | &  \leq C|x-x^{\prime}|.
\end{align*}

\begin{proposition}
\label{continuity} The function $u:[0,T]\times\mathbb{R}^{n}\rightarrow
\mathbb{R}^{k}$ is continuous.
\end{proposition}

\begin{proof}
It remains to prove that $u$ is continuous in $t$. For simplicity, we assume
that $g\equiv0$. For any $t\geq0$, we define
\[
Y_{s}^{t,x}:=Y_{t}^{t,x},\ X_{s}^{t,x}:=x,\ A_{s}^{t,x}:=0,\ Z_{s}%
^{t,x}:=0,\ l(s,X_{s}^{t,x}):=l(t,x),\text{ for }s\in\lbrack0,t].
\]
For $1\leq i\leq k$, it is easy to check that $Y^{t,x;i}$ is the first component of solution to 
reflected $G$-BSDE on the interval $[0,T]$ with terminal value $\phi^{i}%
(X_{T}^{t,x;i})$, generator $\widetilde{f}^{t,x;i}$ and obstacle process
$\widetilde{S}^{t,x;i}$, where
\[
\widetilde{f}^{t,x;i}(s,y,z)=I_{[t,T]}(s)f^{i}(s,X_{s}^{t,x}%
,y,z),\ \widetilde{S}_{s}^{t,x;i}=l(t,x)I_{[0,t)}(s)+l(s,X_{s}^{t,x}%
)I_{[t,T]}(s).
\]
Set
\begin{align*}
&  I_{s}^{t,x}=\widetilde{l}(t,x)I_{[0,t)}(s)+\widetilde{l}(s,X_{s}%
^{t,x})I_{[t,T]}(s),\\
&  b_{s}^{t,x}=(\partial_{s}\widetilde{l}(s,X_{s}^{t,x})+\langle
b(s,X_{s}^{t,x}),D_{x}\widetilde{l}(s,X_{s}^{t,x})\rangle)I_{[t,T]}(s),\\
&  \sigma_{s}^{t,x}=\langle\sigma(s,X_{s}^{t,x}),D_{x}\widetilde{l}%
(s,X_{s}^{t,x})\rangle I_{[t,T]}(s)
\end{align*}
For $0\leq t_{1}\leq t_{2}\leq T$, by Proposition \ref{difference of Y}, we
have
\begin{equation}%
\begin{split}
&  |u(t_{1},x)-u(t_{2},x)|^{2}=|Y_{0}^{t_{1},x}-Y_{0}^{t_{2},x}|^{2}\\
\leq &  C\hat{\mathbb{E}}[|\phi(X_{T}^{t_{1},x})-\phi(X_{T}^{t_{2},x}%
)|]+C\hat{\mathbb{E}}[\int_{0}^{T}|\widetilde{f}^{t_{1},x}(s,Y_{s}^{t_{1},x},Z_{s}%
^{t_{1},x})-\widetilde{f}^{t_{2},x}(s,Y_{s}^{t_{1},x},Z_{s}^{t_{1},x})|^{2}ds](=:\textrm{I})\\
&  +C(\hat{\mathbb{E}}[\sup_{s\in\lbrack0,T]}|\widetilde{S}_{s}^{t_{1},x}-\widetilde{S}_{s}^{t_{2}%
,x}|^{2}])^{1/2}\Psi_{0,T}^{1/2}(=:\textrm{II}),
\end{split}
\label{e8}%
\end{equation}
where
\[
\Psi_{0,T}=\sum_{i=1}^{2}\hat{\mathbb{E}}[|\phi(X_{T}^{t_{i},x})|^{2}%
+\sup_{s\in\lbrack0,T]}|I_{s}^{t_{i},x}|^{2}+\int_{0}^{T}|\widetilde{f}%
^{t_{i},x}(s,\mathbf{0},0)|^{2}+|b_{s}^{t_{i},x}|^{2}+|\sigma_{s}^{t_{i}%
,x}|^{2}+|Y_{s}^{t_{i},x}|^{2}ds].
\]
Similar as the proof of Proposition 4.5 in \cite{Liu-stochastics}, we have
\[
\textrm{I}\leq C(1+|x|^{2})(t_{2}-t_{1}).
\]
Applying Proposition \ref{the1.17} yields that $\Psi_{0,T}\leq C(1+|x|^{2})$.
We directly calculate that
\begin{align*}
&  \hat{\mathbb{E}}[\sup_{s\in\lbrack0,T]}|S_{s}^{t_{1},x}-S_{s}^{t_{2}%
,x}|^{2}]\\
\leq &  |\widetilde{l}(t_{1},x)-\widetilde{l}(t_{2},x)|^{2}+\hat{\mathbb{E}%
}[\sup_{s\in\lbrack t_{1},t_{2}]}|\widetilde{l}(s,X_{s}^{t_{1},x}%
)-\widetilde{l}(t_{2},x)|^{2}]+\hat{\mathbb{E}}[\sup_{s\in\lbrack t_{2}%
,T]}|\widetilde{l}(s,X_{s}^{t_{1},x})-\widetilde{l}(s,X_{s}^{t_{2},x})|^{2}]\\
\leq &  2\sup_{s\in\lbrack t_{1},t_{2}]}|\widetilde{l}(s,x)-\widetilde{l}%
(t_{2},x)|^{2}+\hat{\mathbb{E}}[\sup_{s\in\lbrack t_{1},t_{2}]}|\widetilde{l}%
(s,X_{s}^{t_{1},x})-\widetilde{l}(s,x)|^{2}]+C\hat{\mathbb{E}}[\sup
_{s\in\lbrack t_{2},T]}|X_{s}^{t_{1},x}-X_{s}^{t_{2},x}|^{2}]\\
\leq &  2\sup_{s\in\lbrack t_{1},t_{2}]}|\widetilde{l}(s,x)-\widetilde{l}%
(t_{2},x)|^{2}+C\hat{\mathbb{E}}[\sup_{s\in\lbrack t_{1},t_{2}]}|X_{s}%
^{t_{1},x}-x|^{2}]+C\hat{\mathbb{E}}[\sup_{s\in\lbrack t_{2},T]}|X_{s}%
^{t_{2},X_{t_{2}}^{t_{1},x}}-X_{s}^{t_{2},x}|^{2}]\\
\leq &  2\sup_{s\in\lbrack t_{1},t_{2}]}|\widetilde{l}(s,x)-\widetilde{l}%
(t_{2},x)|^{2}+C(1+|x|^{2})(t_{2}-t_{1})+C\hat{\mathbb{E}}[|X_{t_{2}}%
^{t_{1},x}-x|^{2}]\\
\leq &  2\sup_{s\in\lbrack t_{1},t_{2}]}|\widetilde{l}(s,x)-\widetilde{l}%
(t_{2},x)|^{2}+C(1+|x|^{2})(t_{2}-t_{1}).
\end{align*}
Then, combining all the above analysis, we obtain the desired result.
\end{proof}

\begin{remark}
\upshape{The continuity property of $u$ only needs Assumptions (Ai)-(Aiii). The Assumption (Aiv) is used to guarantee that $u^i$, $1\leq i\leq k$ can be approximated by certain monotone sequences in the proof of Theorem \ref{Rep-M} below.}\end{remark}

Consider the following system of parabolic PDEs with obstacle constraints:
\begin{equation}
\label{PDE}%
\begin{cases}
\min\{u^{i}(t,x)-l^{i}(t,x),-\partial_{t} u^{i}(t,x)-F^{i}(D^{2}_{x}
u^{i},D_{x} u^{i}, u,x,t)\}=0, & (t,x)\in(0,T)\times\mathbb{R}^{n},\\
u^{i}(T,x)=\phi^{i}(x), & x\in\mathbb{R}^{n}; 1\leq i\leq k,
\end{cases}
\end{equation}
where
\begin{align*}
F^i(A,p,r,x,t):=&G(\sigma^T(t,x)A\sigma(t,x)+2\langle p,h(t,x)\rangle+2 g^i(t,x,r,\langle\sigma(t,x),p\rangle))\\
&+\langle b(t,x),p\rangle+f^i(t,x,r,\langle\sigma(t,x),p\rangle),\\
&\textrm{for } (A,p,r,x,t)\in \mathbb{S}(n)\times \mathbb{R}^n\times \mathbb{R}^k\times\mathbb{R}^n\times[0,T],
\end{align*}
and $\mathbb{S}(n)$ is the set of symmetric $n\times n$-matrices. 

In the following, we will show that the value function defined in
\eqref{valuefunction} is the unique solution to the above PDE. Note that
$u^{i}$ is continuous but may  be not differentiable. Therefore, when we call
$u^{i}$ a solution to \eqref{PDE}, it means a viscosity solution. Let us first
introduce the definition of viscosity solution to the PDE \eqref{PDE}, which
needs the following definitions of sub-jets and super-jets. For more details,
we may refer to the paper \cite{CIL}.

\begin{definition}
Let $u\in C((0,T)\times\mathbb{R}^{n})$ and $(t,x)\in(0,T)\times\mathbb{R}%
^{n}$. We denote by $\mathcal{P}^{2,+} u(t,x)$ $($the ``parabolic superjet" of
$u$ at $(t,x)$$)$ the set of triples $(p,q,X)\in\mathbb{R}\times\mathbb{R}%
^{n}\times\mathbb{S}(n)$ satisfying
\begin{align*}
u(s,y)\leq u(t,x)+p(s-t)+ \langle q,y-x\rangle+\frac{1}{2} \langle
X(y-x),y-x\rangle+o(|s-t|+|y-x|^{2}).
\end{align*}
Similarly,
we define $\mathcal{P}^{2,-} u(t,x)$ $($the ``parabolic subjet" of $u$ at
$(t,x)$$)$ by $\mathcal{P}^{2,-} u(t,x):=-\mathcal{P}^{2,+}(- u)(t,x)$.
\end{definition}


\begin{definition}
Let $u\in C([0,T]\times\mathbb{R}^{n};\mathbb{R}^{k})$. It is called a:
\newline\noindent(i) viscosity subsolution of \eqref{PDE} if for each $1\leq
i\leq k$, $u^{i}(T,x)\leq\phi^{i}(x)$, $x\in\mathbb{R}^{n}$, and at any point
$(t,x)\in(0,T)\times\mathbb{R}^{n}$, for any $(p,q,X)\in\mathcal{P}^{2,+}%
u^{i}(t,x)$,
\[
\min\big(u^{i}(t,x)-l^{i}(t,x), -p-F^{i}(X,q,u(t,x),x,t)\big)\leq0;
\]
(ii) viscosity supersolution of \eqref{PDE} if for each $1\leq i\leq k$, $u^{i}%
(T,x)\geq\phi^{i}(x)$, $x\in\mathbb{R}^{n}$, and at any point $(t,x)\in
(0,T)\times\mathbb{R}^{n}$, for any $(p,q,X)\in\mathcal{P}^{2,-}u^{i}(t,x)$,
\[
\min\big(u^{i}(t,x)-l^{i}(t,x), -p-F(X,q,u(t,x),x,t)\big)\geq0;
\]
(iii) viscosity solution of \eqref{PDE} if it is both a viscosity subsolution and supersolution.
\end{definition}

\begin{theorem}\label{Rep-M}
The function $u=(u^{1},\cdots,u^{k})$ defined by \eqref{valuefunction} is the
unique viscosity solution to the system of parabolic PDE \eqref{PDE}.
\end{theorem}

\begin{proof}
For each fixed $(t,x)\in\lbrack0,T]\times\mathbb{R}^{n}$, for any
$m=1,2,\cdots$ and $1\leq i\leq k$, we first consider the following penalized
$G$-BSDEs:
\begin{align*}
Y_{s}^{m,t,x;i}= &  \phi^{i}(X_{T}^{t,x})+\int_{s}^{T}f^{i}(r,X_{r}%
^{t,x},Y_{r}^{m,t,x},Z_{r}^{m,t,x;i})dr+\int_{s}^{T}g^{i}(r,X_{r}^{t,x}%
,Y_{r}^{m,t,x},Z_{r}^{m,t,x;i})d\langle B\rangle_{r}\\
&  +m\int_{s}^{T}(Y_{r}^{m,t,x;i}-l^{i}(r,X_{r}^{t,x}))^{-}dr-\int_{s}%
^{T}Z_{r}^{m,t,x;i}dB_{r}-(K_{T}^{m,t,x;i}-K_{s}^{m,t,x;i}).
\end{align*}
By Theorem 4.7 in \cite{Liu-stochastics}, the function $u^{m}=(u^{m;1}%
,\cdots,u^{m;k})$ defined as
\[
u^{m,i}(t,x):=Y_{t}^{m,t,x;i},\ (t,x)\in\lbrack0,T]\times\mathbb{R}%
^{n},\ 1\leq i\leq k,
\]
is the viscosity solution of the  parabolic PDE
\begin{equation}%
\begin{cases}
\partial_{t}u^{m;i}(t,x)+F^{i}(D_{x}^{2}u^{m;i},D_{x}u^{m;i},u^{m}
,x,t)\\
\ \ \ \ \ \ \ \ \ \ \ \ \ \ \ +m(u^{m;i}(t,x)-l^{i}(t,x))^{-}=0, & (t,x)\in(0,T)\times\mathbb{R}^{n}, 1\leq i\leq k,\\
u^{m;i}(T,x)=\phi^{i}(x), & x\in\mathbb{R}, 1\leq i\leq k.
\end{cases}
\label{penalizedPDE}%
\end{equation}
Besides, by the construction of the solution of reflected $G$-BSDEs via
penalization and Theorem \ref{Myth2-3}, we know that for each $1\leq i\leq k$ and $(t,x)\in
\mathbb{R}^{n}$, $u^{m;i}(t,x)$ converges monotonically up to $u^{i}(t,x)$ as
$m$ approaches infinity.  Recalling
Proposition 4.5 in \cite{Liu-stochastics} and  Proposition \ref{continuity}, $u^{m;i},u^{i}\in C([0,T]\times
\mathbb{R}^{n})$. Thus it follows from the Dini theorem that $u^{m;i}$ converges
uniformly to $u^{i}$ on compact sets.

We first prove that $u$ is a viscosity subsolution of \eqref{PDE}. Let $(t,x)$
be a point such that $u^{i}(t,x)>l^{i}(t,x)$ and let $(p,q,X)\in
\mathcal{P}^{2,+}u^{i}(t,x)$. By Lemma 6.1 in \cite{CIL}, we may find
sequences
\[
m_{j}\rightarrow\infty, \ (t_{j},x_{j})\rightarrow(t,x), \ (p_{j},q_{j}%
,X_{j})\in\mathcal{P}^{2,+}u^{m_{j};i}(t_{j},x_{j}),
\]
such that $(p_{j},q_{j},X_{j})\rightarrow(p,q,X)$. Note that $u^{m}$ is a
viscosity solution to Equation \eqref{penalizedPDE}, hence a subsolution. We
have, for any $j$,
\[
-p_{j}-F^{i}(X_{j},q_{j}, u^{m_{j}}(t_{j},x_{j}),x_{j},t_{j})-m_{j}%
(u^{m_{j};i}(t_{j},x_{j})-l^{i}(t_{j},x_{j}))^{-}\leq0.
\]
Since $(t,x)$ is a point at which $u^{i}(t,x)>h^{i}(t,x)$, it follows from the
uniform convergence of $u^{m;i}$ that $u^{m_{j};i}(t_{j},x_{j})>l^{i}%
(t_{j},x_{j})$ for large enough $j$. Letting $j$ go to infinity in the above
inequality, we obtain
\[
-p-F^{i}(X,q, u(t,x),x,t)\leq0,
\]
which implies that $u$ is a subsolution of \eqref{PDE}.

It remains to prove that $u$ is a supersolution of \eqref{PDE}. For any $1\leq
i\leq k$ and $(t,x)\in[0,T]\times\mathbb{R}^{n}$, let $(p,q,X)\in
\mathcal{P}^{2,-}u^{i}(t,x)$. It is obvious that $u^{i}(t,x)\geq l^{i}(t,x)$.
Then it is sufficient to prove that
\[
-p-F^{i}(X,q, u(t,x),x,t)\geq0.
\]
By a similar analysis as above, there exist sequences 
\[
m_{j}\rightarrow\infty, \ (t_{j},x_{j})\rightarrow(t,x), \ (p_{j},q_{j}%
,X_{j})\in\mathcal{P}^{2,-}u^{m_{j};i}(t_{j},x_{j}),
\]
such that $(p_{j},q_{j},X_{j})\rightarrow(p,q,X)$. Since $u^{m}$ is a
viscosity solution to Equation \eqref{penalizedPDE}, hence a supersolution, we
have, for any $j$,
\[
-p_{j}-F^{i}(X_{j},q_{j}, u^{m_{j}}(t_{j},x_{j}),x_{j},t_{j})-m_{j}%
(u^{m_{j};i}(t_{j},x_{j})-l^{i}(t_{j},x_{j}))^{-}\geq0.
\]
Letting $j$ approach infinity, we get the desired result.

The uniqueness of viscosity solutions to \eqref{PDE} follows from classical
arguments. See, e.g., \cite{CIL} and \cite{TY}. So we omit it and the proof is complete.
\end{proof}





\section*{Acknowledgments}
\noindent 
The authors would like to thank
Professor Peng Luo for helpful discussions. The authors also thank the editor and the referee for useful suggestions that improved the first version of the paper.
Li's work was supported  by the National Natural Science Foundation of China (No. 12301178), the Natural Science Foundation of Shandong Province for Excellent Young Scientists Fund Program (Overseas) (No. 2023HWYQ-049) and  the Qilu Young Scholars Program of Shandong University. Liu's work was supported by National Natural Science
Foundation of China (No. 12201315) and the Fundamental Research Funds for the
Central Universities, Nankai University (No. 63221036).










\section*{Declarations}
\textbf{Conflict of Interest}\ \ 
The authors declared that they have no conflict of interest.
\\
\\
\textbf{Data Availability Statement}\ \  Data sharing not applicable to this article as no datasets were generated or analyzed during
the current study.



\begin{thebibliography}{99}

\bibitem {BCFE}Bally V., Caballero M. E., Fernandez B., El Karoui N. Reflected
BSDEs, PDEs and variational inequalities. preprint INRIA, 2002.





\bibitem {CIL}Crandall M., Ishii H., Lions P. L. User's guide to the viscosity
solutions of second order partial differential equations. Bull. Amer. Math.
Soc., 1992, 27: 1-67.



\bibitem {DHP11}Denis L., Hu M., Peng S. Function spaces and capacity related
to a sublinear expectation: application to $G$-Brownian motion pathes.
Potential Anal., 2011, 34: 139-161.




\bibitem {KKPPQ}El Karoui N., Kapoudjian C., Pardoux E., Peng S., Quenez M. C.
{Reflected solutions of backward SDE's, and related obstacle problems for
PDE's}. Ann. Probab., 1997, 23(2): 702-737.

\bibitem {EPQ}El Karoui N., Pardoux E., Quenez M. C. {Reflected backward SDE's
and American options}. Numerical Methods in Finance (Cambridge Univ. Press),
1997: 215-231.

\bibitem {EPQ97}El Karoui N., Pardoux E., Quenez M. C. Backward stochastic
differential equations in finance. Math. Finance, 1997, 7: 1-71

\bibitem{GIOOQ} Grigorova M., Imkeller P., Offen E., Ouknine Y., Quenez M.C.  Reflected BSDEs when the obstacle is not right-continuous and optimal stopping. Ann. Appl. Probab., 2017, 27: 172-196.







\bibitem {HJPS1}Hu M., Ji S., Peng S., Song Y. {Backward stochastic
differential equations driven by $G$-Brownian motion}. Stoch. Process. Appl., 2014, 124: 759-784.

\bibitem {HJPS2}Hu M., Ji S., Peng S., Song Y. {Comparison theorem,
Feynman-Kac formula and Girsanov transformation for BSDEs driven by
$G$-Brownian motion}. Stoch. Process. Appl., 2014, 124: 1170-1195.

\bibitem{HTW2022} Hu Y., Peng, S., 
{On the comparison theorem for multidimensional BSDEs}.
C. R. Math. Acad. Sci. Paris, 2006, 343: 135-140.

\bibitem{HTW2022}Hu Y., Tang S., Wang  F., {Quadratic  $G$-BSDEs with convex generators and unbounded terminal conditions}.
Stoch. Process. Appl., 2022, 153: 363-390.





\bibitem {KLQT}Kobylanski M., Lepeltier J. P., Quenez M. C., Torres S. Reflected
BSDE with superlinear quadratic coefficient. Probab. Math. Statist., Fasc. 1, 2002, 22: 51-83.

\bibitem {LX}Lepeltier J. P., Xu M. Penalization method for reflected backward
stochastic differential equations with one r.c.l.l. barrier. Statist. Probab.
Lett., 2005, 75: 58-66.


\bibitem{ll-a}
Li H., Liu G.
Multi-dimensional reflected BSDEs driven by $G$-Brownian motion with diagonal generators.
 arXiv preprint, 2023, arXiv:2310.11376.

\bibitem {lp}Li H., Peng S. Reflected BSDE driven by $G$-Brownian motion with
an upper obstacle. Stoch. Process. Appl., 2020, 130(11): 6556-6579.


\bibitem {LPSH}Li H., Peng S., Soumana Hima A. Reflected Solutions of backward
stochastic differential equations Driven by $G$-Brownian Motion. Sci. China Math., 2018, 61(1): 1-26.

\bibitem {LS}Li H., Song Y. Backward stochastic differential equations driven
by $G$-Brownian motion with double reflections. J. Theor. Probab., 2021,  34: 2285–2314.

\bibitem {LP}Li X., Peng S. Stopping times and related It\^o's calculus with
$G$-Brownian motion. Stoch. Process. Appl., 2011, 121: 1492-1508.

\bibitem {Liu-stochastics}Liu G. Multi-dimensional BSDEs driven by
$G$-Brownian motion and related system of fully nonlinear PDEs. Stochastics,
2020, 92(5): 659-683.








\bibitem {PP}Pardoux E., Peng S. Adapted solutions of backward equations.
Syst. Control Lett., 1990, 14: 55-61.

\bibitem {PP92}Pardoux E., Peng S. Backward stochastic differential equations
and quasilinear parabolic partial differential equations. Stochastic Partial
Differential Equations and their Applications, 1992, Proc. IFIP, LNCIS 176: 200-217.

\bibitem {P93}Peng S. Backward stochastic differential equations and
applications to optimal control. Appl. Math. Optim., 1993, 27(2): 125-144.

\bibitem {P07a}Peng S. $G$-expectation, $G$-Brownian Motion and Related
Stochastic Calculus of It\^o type. Stochastic analysis and applications, Abel
Symp., 2, Springer, Berlin, 2007: 541-567.

\bibitem {P08a}Peng S. Multi-dimensional $G$-Brownian motion and related
stochastic calculus under $G$-expectation. Stoch. Process. Appl., 2008, 118(12): 2223-2253.

\bibitem {P10}Peng S. Nonlinear expectations and stochastic calculus under uncertainty. 2019. Springer-Verlag Berlin
Heidelberg.

\bibitem {PX}Peng S., Xu M. The smallest $g$-supermartingale and reflected BSDE with single and double $L^{2}$ obstacles. Ann. I. H. Poincare-PR, 2005, 41: 605-630.





\bibitem {S11}Song Y. Some properties on $G$-evaluation and its applications
to $G$-martingale decomposition. Sci. China Math., 2011, 54: 287-300.


\bibitem {TY}Tang S., Yong J. Finite horizon stochastic optimal switching and
impulse controls with a viscosity solution approach. Stochastics Stochastics
Rep. 45 (1993), no. 3-4, 145-176.

\bibitem {WX}Wu Z., Xiao H. Multi-dimensional reflected backward stochastic
differential equations and the comparison theorem. Acta Math. Sci. Ser. B (Engl. Ed.),
2010, 30(5): 1819-1836.

\end{thebibliography}
\end{document}